\documentclass[11pt,a4paper]{article}
\usepackage[nottoc]{tocbibind}
\usepackage{graphicx}
\usepackage{enumerate}
\usepackage{amsmath}
\usepackage{amsfonts}
\usepackage{psfrag}
\usepackage{color}
\usepackage{pgf,tikz}
\usepackage{mathrsfs}
\usetikzlibrary{arrows}
\usepackage{fancyhdr}
\usepackage{txfonts} 

\usepackage{hyperref}
\hypersetup{colorlinks=true,linkcolor=blue,citecolor=red}

\newtheorem{theorem}{Theorem}

\newtheorem{corollary}[theorem]{Corollary}
\newtheorem{definition}{Definition}
\newtheorem{remark}{Remark}

\newtheorem{lemma}{Lemma}
\newtheorem{proposition}{Proposition}

\newcommand{\mint}{\mathop{\int\hskip -1,05em -\, \!\!\!}\nolimits}

\numberwithin{equation}{section}
\allowdisplaybreaks[4]

\newenvironment{proof}{\smallskip\noindent\emph{Proof.}\hspace{1pt}}%
{\hspace{-5pt}{\nobreak\quad\nobreak\hfill\nobreak$\square$\vspace{8pt}%
\par}\smallskip\goodbreak}

\def\eqn#1$$#2$${\begin{falign}\label#1#2\end{falign}}

\newcommand\eps\varepsilon
\def\eqn#1$$#2$${\begin{equation}\label#1#2\end{equation}}

\newcommand{\be}{\begin{equation}}
\newcommand{\ee}{\end{equation}}

\newcommand{\supp}{\operatorname{supp}}

\newcommand{\const}{\operatorname{const}}
\newcommand{\snr}[1]{\lvert #1\rvert}

\newcommand{\nr}[1]{\lVert #1 \rVert}

\newcommand{\RN}{\mathbb{R}^{N}}

\newcommand{\N}{\mathbb{N}}

\def\name[#1, #2]{#1 #2}

\newcommand{\rif}[1]{(\ref{#1})}

\newcommand{\cd}[1]{\textcolor[rgb]{0.00,0.00,0.00}{#1}}
\newcommand{\oh}[1]{\textcolor[rgb]{0.00,0.00,0.00}{#1}}

\definecolor{ffqqqq}{rgb}{1.,0.,0.}
\definecolor{uuuuuu}{rgb}{0.26666666666666666,0.26666666666666666,0.26666666666666666}

\DeclareMathOperator{\diam}{diam}

\begin{document}
\title{Regularity for \oh{multi-phase} variational problems}

\author{
\textsc{Cristiana De Filippis\thanks{Mathematical Institute, University of Oxford, Andrew Wiles Building, Radcliffe Observatory Quarter, Woodstock Road, Oxford, OX26GG, Oxford, United Kingdom. E-mail:
      \texttt{Cristiana.DeFilippis@maths.ox.ac.uk}} \  \& Jehan Oh \thanks{\oh{Fakult\"at f\"ur Mathematik}, Universit\"at Bielefeld, Postfach 100131, D-33501 Bielefeld. E-mail: \texttt{\oh{joh@math.uni-bielefeld.de}}}}}

\date{\today
}

\maketitle
\thispagestyle{plain}

\vspace{-0.6cm}
\begin{abstract}
We prove $C^{1,\nu}$ regularity for local minimizers of the \oh{multi-phase} energy:
\begin{flalign*}
w \mapsto \int_{\Omega}\snr{Dw}^{p}+a(x)\snr{Dw}^{q}+b(x)\snr{Dw}^{s} \ dx,
\end{flalign*}
under sharp assumptions relating the couples $(p,q)$ and $(p,s)$ to the H\"older exponents of the modulating coefficients $a(\cdot)$ and $b(\cdot)$,  respectively.
\end{abstract}

       \vspace{5pt}                           %
\vspace{0.2cm}

\setlength{\voffset}{-0in} \setlength{\textheight}{0.9\textheight}

\setcounter{page}{1} \setcounter{equation}{0}


\section{Introduction and results}
The aim of this paper is to analyze the regularity properties of non-autonomous variational integrals
of the type 
\eqn{tipo}
$$
w \mapsto \int_{\Omega}F(x, Dw) \ dx\;,
$$
where $\Omega\subset \mathbb{R}^{n}$ is \oh{a bounded open domain with $n\ge 2$},
and emphasize a few new phenomena emerging when considering non-uniformly elliptic operators. Let us briefly review the situation. In the case of functionals 
satisfying standard polynomial growth and ellipticity of the type
\eqn{tipo1}
$$
F(x, Dw) \approx |Dw|^p \qquad \mbox{and} \qquad \partial_{zz} F(x, Dw) \approx |Dw|^{p-2} Id\;,
$$
as for instance
\eqn{tipo2}
$$
w \mapsto \int_{\Omega}a(x)|Dw|^p \ dx\;, \qquad 0 < \nu \leq a(x) \leq L\;, 
$$
the regularity of minimizers is \oh{well-understood}. In particular, assuming that the partial function $x \mapsto F(x, \cdot)$ is H\"older continuous with some exponent (\oh{for example, $a(\cdot)$ is locally H\"older continuous in the case of \rif{tipo2}}), then it turns out that the gradient of minima is locally H\"older continuous. This is a well established theory, both in the scalar and in the vectorial case, for which we refer for instance to \cite{KM1, KM2, Manth, dark}. The situation drastically changes when considering non-uniformly elliptic functionals. These are functionals of the type in where the ellipticity ratio
$$
\mathcal R(z, B) \coloneqq \frac{\sup_{x \in B}\,  \mbox{highest eigenvalue of}\ \partial_{zz}  F(x,z)}{\inf_{x \in B}\,\mbox{lowest eigenvalue of}\  \partial_{zz} F(x,z)} \;, 
$$
where $B\subset \Omega$ is a ball, might become unbounded with $|z|$. This is the case, for instance, of the double phase functional given by
\begin{flalign}\label{hmvp0}
W^{1,H(\cdot)}(\Omega)\ni w\mapsto \oh{\int_{\Omega}\snr{Dw}^{p}+a(x)\snr{Dw}^{q} \ dx}\;.
\end{flalign}
This functional has been introduced by Zhikov in the context of Homogenization and its integrand changes its growth - from $p$ to $q$-rate - depending on the fact that $x$ belongs to $\{a(\cdot)=0\}$ or not (here is where the terminology double phase stems from). In the first case we have, following a terminology introduced in \cd{\cite{double}}, the $p$-phase, in the other we have the $(p,q)$-phase. In the case of \rif{hmvp0}, the regularity of minimizers is regulated by a subtle interaction between the pointwise behaviour of the partial function $x \mapsto F(x, \cdot)$ and the growth assumption satisfied by $z \mapsto F(\cdot, z)$. For instance, as established in the work of Baroni, Colombo and Mingione \cite{bcm1, bcm2, bcm, double, colmin}, \oh{sufficient and necessary conditions} for regularity of minimizers of \oh{the functional \rif{hmvp0} are} that
\eqn{cond-double}
$$
a(\cdot) \in C^{0, \alpha}(\Omega) \qquad \mbox{and} \qquad \frac qp \leq 1 + \frac{\alpha}{n}\;.
$$
Specifically, if \rif{cond-double} holds, then minimizers of \oh{the functional \rif{hmvp0}} are locally $C^{1,\beta}$, for some $\beta >0$, otherwise, they can be even discontinuous; see also \cite{colmin2, sharp, ffm}. After these contributions, functionals with double phase type have become a topic of intense study, see for instance \cite{byun1, byunoh1, byunoh2, hasto, hasto1, ok1, ok2}.

The condition in \rif{cond-double} plays a role also when considering more general functionals of the type in \rif{tipo},under so called $(p,q)$-growth conditions, i.e.:
$$
|Dw|^p \lesssim F(x, Dw) \lesssim |Dw|^q \qquad \mbox{and} \qquad |Dw|^{p-2} Id\lesssim  \partial_{zz} F(x, Dw) \lesssim |Dw|^{q-2} Id\;.
$$
For this we refer to \cite{carkrispas, elemarmas, sharp}. Moreover, it intervenes in the validity of a corresponding Calder\'on-Zygmund theory \cite{colmin2}. We refer to the papers of Marcellini \cite{mar1, mar2, mar3} for more on general functionals with $(p,q)$-growth. 

The aim of this paper is to study a significant generalization of \oh{the functional \rif{hmvp0}}, considering a functional that exhibit three phases. We shall indeed consider the following multiphase Multi-Phase variational energy
\begin{flalign}\label{hmvp}
W^{1,H(\cdot)}(\Omega)\ni w\mapsto\mathcal{H}(w,\Omega)=\int_{\Omega}H(x, \oh{Dw}) \ dx,
\quad 1<p\cd{<} q \cd{\le} s\;, 
\end{flalign}
with
\begin{flalign}\label{hm}
H(x,z):=\snr{z}^{p}+a(x)\snr{z}^{q}+b(x)\snr{z}^{s},
\end{flalign}
and where the functions $a(\cdot)$ and $b(\cdot)$ satisfy \oh{the following assumptions}
\eqn{mathAB}
$$
a\in C^{0,\alpha}(\Omega), \ \ a(\cdot)\ge 0, \ \ \alpha \in (0,1], \qquad 
 b\in C^{0,\beta}(\Omega), \ \ b(\cdot)\ge 0, \ \ \beta \in (0,1]. 
$$
\cd{As not to trivialize the problem, we specifically focus on the case in which the strict inequality $\eqref{hmvp}_{2}$ holds.} The analysis of this functional then opens the way to that of functional exhibiting an arbitrary number of phases, and involves several subtle points. The main one can be described as follows. In the double phase case of \oh{the functional \rif{hmvp0}} the main game is to control the interaction between the potentially degenerate parte of the energy $a(x)|Dw|^q$ (here degenerate means that it can be $a(x)\equiv 0$) with the non-degenerate one $|Dw|^p$, that always provides a solid rate of ellipticity. This is done in \cite{bcm, double, colmin} via a careful comparison scheme \cd{built in} order to distinguish between the two phases. Here the situation changes and the game becomes more delicate. Indeed, the problem is to control the interaction between the two possibly degenerate parts of the energy, that is $a(x)|Dw|^q$ and $b(x)|Dw|^s$. A new aspect in fact emerges here. We see that, in presence of a finer structure, conditions of the type in \rif{cond-double} can be in a sense relaxed. In fact, an immediate application of \rif{cond-double} would \oh{provide} us with the conditions $a,b\in C^{0,\alpha}(\Omega)$ with $s/p \leq 1 +\alpha/n$, by considering the global regularity of $x \mapsto F(x, \cdot)$. Instead, we see that the new condition coming into the play takes into account more precisely the way the presence of $x$ affects the growth with respect to the gradient variable. Specifically, we shall assume that 
\begin{flalign}\label{rat}
\frac{q}{p}\le1+\frac{\alpha}{n}\quad \mathrm{and} \quad \frac{s}{p}\le 1+\frac{\beta}{n}.
\end{flalign}
In other words, less regularity is needed on the coefficient affecting the $q$-growth, intermediate part of the energy density. Our main results is indeed the following main result of the paper (see \oh{the next section} for more definitions and notation):
\begin{theorem}[$C^{1,\nu}$-local regularity]\label{T1}
Let $u$ be a local minimizer of \oh{the functional \eqref{hmvp} under} assumptions \eqref{mathAB} and \eqref{rat}. Then there exists $\nu=\nu(\texttt{data})\in (0,1)$ such that $u \in C^{1,\nu}_{\mathrm{loc}}(\Omega)$.
\end{theorem}
We remark that the sharpness of both conditions in \rif{rat} can be obtained by the same counterexamples in \cite{sharp, ffm}. Moreover, as it is well-known from the \oh{regularity theory for the standard $p$-Laplacean}, the one in Theorem \ref{T1} is the maximal regularity obtainable for $u$. 

A worth singling-out intermediate \oh{result} towards the proof of Theorem \ref{T1} is the following intrinsic Morrey decay estimate, which reduces to a classical estimate in the case of the $p$-Laplacean and that extends to the multi phase case the one proved in \cite{bcm,double, colmin} for minima of functionals with a double phase. 
\begin{theorem}[Intrinsic Morrey Decay]\label{T0}
Let $u$ be a local minimizer of \oh{the functional \eqref{hmvp} under} assumptions \eqref{mathAB} and \eqref{rat}. Then, for every $\vartheta \in (0,n)$, there exists a positive constant $c=c(\texttt{data}(\Omega_{0}),\vartheta)$ such that the decay estimate
\begin{flalign}\label{mor}
\int_{B_{\rho}}H(x,Du) \ dx \le c\left(\frac{\rho}{r}\right)^{n-\vartheta}\int_{B_{r}}H(x,Du) \ dx
\end{flalign}
holds whenever $B_{\rho}\subset B_{r}\Subset \Omega_{0}$ are concentric balls with \oh{$0< \rho \le r \le 1$}.
\end{theorem}
Let us quickly describe the techniques we are employing to obtain the aforementioned theorems. The starting point is the recent proof of regularity of minimizers of double-phase variational problems \cd{appeared} in \cite{bcm}, and based on a suitable use of harmonic type approximations lemmas (see also \cite{colmin} for a first version). This is just a general blueprint we move from to treat the the real new difficulty here. Indeed, as we are dealing here with the presence of several phase transitions, and we have to carefully handle the regularity of solutions on the zero sets $\{a(x)=0\}$ and $\{b(x)=0\}$, that is, when the functional tends to loose part of its ellipticity properties and switch their kind of ellipticity. Therefore we have to handle the presence of two different transitions. We come up with a delicate scheme of alternatives and of nested exit time arguments, carefully controlling the interaction between the two phase transitions. It is then clear that the techniques introduced in this paper allow to prove regularity results for functionals with an arbitrary large numbers of \oh{phases, for instance,}
$$
w \mapsto \int_{\Omega} \left[|Dw|^p + \sum_{i=1}^m a_i(x)|Du|^{p_i}\right] \, dx \,, \quad \oh{1<} p \cd{<} p_1 \le \oh{\cdots }\le p_m
$$
and $a_i(\cdot)\in \oh{C^{0, \alpha_i} (\Omega)}$.  
\section{Notation and preliminaries}\label{2}
In this section we establish some basic notation that we are going to use for the rest of the paper. 
As in the Introduction, 
\oh{$\Omega$ will denote an open subset of $\mathbb{R}^{n}$ with $n\ge 2$}. As usual, we shall denote by $c$ a general constant larger than one, which can vary from line to line. Relevant dependencies from certain parameters will be emphasized using brackets, i.e.: $c=c(n,p,q,s)$ means that $c$ depends on $n,p,q,s$. We denote with $B_{r}(x_{0})=\left\{x \in \mathbb{R}^{n}\colon \ \snr{x-x_{0}}<r\right\}$ the $n$-dimensional open ball centered at $x_{0}$ and with radius $r>0$; when non relevant or clear from the context, we will omit to indicate the centre as follows: $B_{r}=B_{r}(x_{0})$. When not differently specified, in the same context, balls with different radius will share the same center. If $A\subset \mathbb{R}^{n}$ is any measurable subset with finite and positive Lebesgue's measure $\snr{A}>0$ and $f\colon A \oh{\to} \mathbb{R}^{N}$, $N\ge 1$ is a measurable map, we shall denote its integral average over $A$ as
\begin{flalign*}
\oh{(f)_{A} = \mint_{A}f(x) \ dx=\frac{1}{\snr{A}}\int_{A}f(x) \ dx.}
\end{flalign*}
\oh{When $A=B_r$, we shall write
\begin{flalign*}
(f)_{r} := (f)_{B_r} = \mint_{B_r}f(x) \ dx=\frac{1}{\snr{B_r}}\int_{B_r}f(x) \ dx.
\end{flalign*}}
The integrand $H(\cdot)$ has already been defined in 
\eqref{hm}. With abuse of notation we shall denote $H(x,z)$ when $z \in \mathbb{R}^{n}$ and when $z \in \mathbb{R}$, that is when $z$ is a scalar, so that we shall intend both $H \colon \Omega \times \mathbb{R}^{n} \to [0, \infty)$ and 
$H \colon \Omega \times \mathbb{R} \to [0, \infty)$. The modulating coefficients $a(\cdot)$ and $b(\cdot)$ will always satisfy \eqref{mathAB}. Here we recall that, if $f\colon \Omega \oh{\to} \mathbb{R}$ is any $\gamma$-H\"older continuous map \oh{with} $\gamma \in (0,1)$ and $A \subset \Omega$, then its H\"older seminorm is defined as
\begin{flalign*}
[f]_{0,\gamma;A}=\sup_{x,y\in A, \ x \not = y}\frac{\snr{f(x)-f(y)}}{\snr{x-y}^{\gamma}}, \quad [f]_{0,\gamma}=[f]_{0,\gamma;\Omega}.
\end{flalign*}
We are going to use several tools from the Orlicz space setting, therefore we start with the following preliminaries.
\begin{definition}\label{defy}
A function $\varphi\colon [0,\infty)\to [0,\infty)$ is said to be \oh{a Young function} if it satisfies the following conditions: $\varphi(0)=0$ and there exists the derivative $\varphi'$, which is right-continuous, non decreasing and satisfies
$$
\oh{ \varphi'(0)=0, \quad \varphi'(t)>0 \quad \mbox{for} \ \, t>0, \quad \mbox{and} \quad \lim_{t\to \infty}\varphi'(t)=\infty.}
$$
\end{definition}
\begin{remark}\label{remy}
\emph{In order to extrapolate good regularity properties for minimizers of functionals with \oh{$\varphi$-growth}, we need to assume something more. Precisely, from now on, in addition to the basic assumptions listed in Definition \ref{defy} we will also suppose that $\varphi \in C^{1}[0,\infty)\cap C^{2}(0,\infty)$ and that
\eqn{uniphi}
$$i_{\varphi} \leq \frac{t\varphi''(t)}{\varphi'(t)} \leq s_{\varphi}\qquad \mbox{unifornly in $t$}\,.$$
This is equivalent to the so-called $\Delta_{2}$ condition, since $t\mapsto \varphi(t)$ is non decreasing, see \cite{diestrver}, Section 2. }
\end{remark}

\begin{definition}
Let $\varphi$ be a Young function in the sense of Definition \ref{defy} and Remark \ref{remy}. Given $\Omega \subset \mathbb{R}^{n},$ the Orlicz space $L^{\varphi}(\Omega)$ is defined as
\begin{flalign*}
L^{\varphi}(\Omega)=\left \{u\colon \Omega \to \mathbb{R} \ \mathrm{such \ that \ }\int_{\Omega}\varphi(\oh{|u|}) \ dx 
<\infty\right \}
\end{flalign*}
and, consequently,
\begin{flalign*}
W^{1,\varphi}(\Omega)=\left \{u\in W^{1,1}(\Omega)\cap L^{\varphi}(\Omega)\ \mathrm{such \ that \ }Du \in L^{\varphi}(\Omega, \mathbb{R}^{N})\right\}.
\end{flalign*}
The definitions of the variants $W^{1,\varphi}_{0}(\Omega)$ and $W^{1,\varphi}_{\mathrm{loc}}(\Omega)$ come in an obvious way from the one of $W^{1,\varphi}(\Omega)$.
\end{definition}
In connection to $H(\cdot)$, we also consider the following Orlicz-Musielak-Sobolev space
\begin{flalign}\label{hspace}
W^{1,H(\cdot)}(\Omega)=\left\{u \in W^{1,1}\oh{(\Omega)} \colon \ H(\cdot,Du)\in L^{1}(\Omega) \right\},
\end{flalign}
with local variant defined in an obvious way and $W^{1,H(\cdot)}_{0}(\Omega)=W^{1,H(\cdot)}(\Omega)\cap W^{1,p}_{0}(\Omega)$; we refer to \cite{bcm, hasto, hasto1} for more on such spaces. 

For later uses, we introduce also the auxiliary Young functions
\begin{flalign}
\begin{cases}\label{yfs}
H_{0}(z)=\snr{z}^{p}+a_{0}\snr{z}^{q}+b_{0}\snr{z}^{s},\\
H_{0}^{s}(z)=\snr{z}^{p}+b_{0}\snr{z}^{s},\\
H_{0}^{q}(z)=\snr{z}^{p}+a_{0}\snr{z}^{q},\\
H_{0}^{p}(z)=\snr{z}^{p}.
\end{cases}
\end{flalign}
The values of the constants $a_{0},b_{0}\ge 0$ will vary according to the necessities, but all the estimates we eventually get are independent on their value.\\ 
In the following we will often use the vector field
\begin{flalign}\label{vf}
V_{t}(z)=\snr{z}^{(t-2)/2}z, \quad t \in \{p,q,s\}.
\end{flalign}
\oh{We recall from \cite{diestrver}, important features of \eqref{vf}: there exists $c=c(n,t)>0$ such that
\begin{flalign}\label{control_2}
\snr{V_{t}(z_{1})-V_{t}(z_{2})}^2 &\le c \left( \snr{z_{1}}^{t-2}z_{1}-\snr{z_{2}}^{t-2}z_{2} \right) \cdot \left( z_{1}-z_{2} \right),\\
\snr{V_{t}(z_{1})-V_{t}(z_{2})}&\sim (\snr{z_{1}}+\snr{z_{2}})^{\frac{t-2}{2}}\snr{z_{1}-z_{2}},\label{control3}
\end{flalign}
where the constants implicit in \eqref{control3} depend only on $n,t$ and, for all $z \in \mathbb{R}^{n}$
\begin{flalign}\label{control}
\snr{V_{t}(z)}^{2}=\snr{z}^{t} \quad .
\end{flalign}
For later uses, we introduce the following auxiliary functions
\begin{flalign}
\begin{cases}\label{vfs}
\mathcal{V}_{0}(z_{1},z_{2})^{2}=\snr{V_{p}(z_{1})-V_{p}(z_{2})}^{2}+a_{0}\snr{V_{q}(z_{1})-V_{q}(z_{2})}^{2}+b_{0}\snr{V_{s}(z_{1})-V_{s}(z_{2})}^{2},\\
\mathcal{V}_{0}^{s}(z_{1},z_{2})^{2}=\snr{V_{p}(z_{1})-V_{p}(z_{2})}^{2}+b_{0}\snr{V_{s}(z_{1})-V_{s}(z_{2})}^{2},\\
\mathcal{V}_{0}^{q}(z_{1},z_{2})^{2}=\snr{V_{p}(z_{1})-V_{p}(z_{2})}^{2}+a_{0}\snr{V_{q}(z_{1})-V_{q}(z_{2})}^{2},\\
\mathcal{V}_{0}^{p}(z_{1},z_{2})^{2}=\snr{V_{p}(z_{1})-V_{p}(z_{2})}^{2}.
\end{cases}
\end{flalign}
}
Let us also recall some important tools in regularity. The first one is an iteration lemma from \cite{giagiu1}.
\begin{lemma}\label{iter}
Let $h\colon [\rho, R_{0}]\to \mathbb{R}$ be \oh{a non-negative bounded function} and $0<\theta<1$, $0\le A$, $0<\beta$. Assume that
$
h(r)\le A(d-r)^{-\beta}+\theta h(d)
$
for $\rho \le r<d\le R_{0}$. Then
$
h(\rho)\le cA/(R_{0}-\rho)^{-\beta}
$ holds, 
where $c=c(\theta, \beta)>0$.
\end{lemma}
Along the proof we shall make an intensive use of the regularity properties of $\varphi$-harmonic maps, so we recall definition and some reference estimates from \oh{Lemma 5.8 and Theorem 6.4 in \cite{diestrver}}.
\begin{definition}\label{defphih}
Let $U\Subset \Omega$ be an open set and $u_{0} \in W^{1,\varphi}_{\mathrm{loc}}(\Omega,\RN)$ be any function. With $\varphi$-harmonic map, we mean a map $h\in u_{0}+W^{1,\varphi}_{0}(U,\RN)$ solving the Dirichlet problem
\begin{flalign*}
u_{0}+W^{1,\varphi}_{0}(U,\RN)\ni w\mapsto \min\int_{U}\varphi(\snr{Dw}) \ dx.
\end{flalign*}
\end{definition}
\begin{proposition}\label{p0}
Let $\Omega \subset \mathbb{R}^{n}$ be open and $\varphi \in C^{2}(0,\infty)\cap C^{1}[0,\infty)$ be \oh{a Young function} satisfying \eqref{uniphi}. If $h\in W^{1,\varphi}(\Omega, \RN)$ is $\varphi$-harmonic on $\Omega$, then for any ball $B_{r}$ with $B_{2r}\Subset \Omega$ there holds
\begin{flalign*}
\sup_{B_{r}}\varphi(\snr{Dh})\le c\mint_{B_{2r}}\varphi(\snr{Dh}) \ dx,
\end{flalign*}
where $c$ depends only on $n, N, i_{\varphi}, s_{\varphi}$.
\end{proposition}
We conclude this section by giving the definition of a local minimizer of \eqref{hmvp}.
\begin{definition}
A map $u \in W^{1,H(\cdot)}_{\mathrm{loc}}(\Omega)$ is a local minimizer of the variational integral \eqref{hmvp} if and only if $H(\cdot,Du)\in L^{1}_{\mathrm{loc}}(\Omega)$ and the minimality condition $\mathcal{H}(u,\supp(u-v))\le \mathcal{H}(v,\supp(u-v))$ is satisfied whenever $v \in W^{1,1}_{\mathrm{loc}}(\Omega)$ and $\supp(u-v)\subset \Omega$.
\end{definition}
\section{First regularity results}\label{3}
In this section we collect a few basic regularity results which can be proved with minor adjustments to the proofs contained in \cite{bcm, double, colmin, ok1}.
\oh{\begin{lemma}[Sobolev-Poincar\'e inequality]\label{L1}
Let $1<p\cd{<} q\cd{<} s$ and $\alpha, \beta \in (0,1]$ verifying \eqref{mathAB}, \eqref{rat}. Then there exist a constant $c=c(n,p,q,s)$ and an exponent $d=d(n,p,q,s) \in (0,1)$ such that for any $w \in W^{1,H(\cdot)}(B_{r})$ with $r \le 1$,
\begin{flalign}\label{sp}
\mint_{B_{r}}H\left(x, \frac{w-(w)_r}{r} \right) dx \le c\left(1+[a]_{0,\alpha}\nr{Dw}^{q-p}_{L^{p}(B_{r})}+[b]_{0,\beta}\nr{Dw}^{s-p}_{L^{p}(B_{r})}\right)\left(\mint_{B_{r}}H(x,Dw)^{d} \, dx\right)^{\frac{1}{d}}.
\end{flalign}
Furthermore, the same is still true with $w-(w)_r$ replaced by $w$ if we consider $w\in W^{1,H(\cdot)}_{0}(B_{r})$.
\end{lemma}
\begin{proof}
We first consider the case
\begin{equation}\label{poincare_case_0}
\sup_{x \in B_r} a(x) \le 4[a]_{0,\alpha}r^{\alpha} \quad \mbox{and} \quad \sup_{x \in B_r} b(x) \le 4[b]_{0,\beta}r^{\beta}.
\end{equation}
Then it follows from the classical Sobolev-Poincar\'e inequality that
\begin{flalign*}
\mint_{B_r} a(x) \frac{\snr{w-(w)_r}^q}{r^q} \, dx \leq 4[a]_{0,\alpha}r^{\alpha} \mint_{B_r} \frac{\snr{w-(w)_r}^q}{r^q} \, dx \leq c[a]_{0,\alpha}r^{\alpha} \left( \mint_{B_r} \snr{Dw}^{q_*} dx \right)^{\frac{q}{q_*}},
\end{flalign*}
\cd{with $c=c(n,q)$ }and
\begin{flalign*}
\mint_{B_r} b(x) \frac{\snr{w-(w)_r}^s}{r^s} \, dx \leq 4[b]_{0,\beta}r^{\beta} \mint_{B_r} \frac{\snr{w-(w)_r}^s}{r^s} \, dx \leq c[b]_{0,\beta}r^{\beta} \left( \mint_{B_r} \snr{Dw}^{s_*} dx \right)^{\frac{s}{s_*}},
\end{flalign*}
\cd{$c=c(n,s)$}, where
\begin{flalign*}
q_* := \max \left\lbrace \frac{nq}{n+q},1 \right\rbrace, \quad s_* := \max \left\lbrace \frac{ns}{n+s},1 \right\rbrace.
\end{flalign*}
We see from the assumption \eqref{rat} that $q_* < p$ and $s_* < p$.
Therefore, we obtain from H\"older's inequality, \eqref{rat} and the fact $r \leq 1$ that
\begin{flalign}\label{poincare_q}
\mint_{B_r} a(x) \frac{\snr{w-(w)_r}^q}{r^q} \, dx & \leq c[a]_{0,\alpha}r^{\alpha} \left( \mint_{B_r} \snr{Dw}^{p} dx \right)^{\frac{q-p}{p}} \left( \mint_{B_r} \snr{Dw}^{q_*} dx \right)^{\frac{p}{q_*}} \nonumber \\
& \leq c[a]_{0,\alpha}r^{\alpha-\frac{n(q-p)}{p}} \nr{Dw}_{L^{p}(B_{r})}^{q-p} \left( \mint_{B_r} \snr{Dw}^{q_*} dx \right)^{\frac{p}{q_*}} \nonumber \\
& \leq c[a]_{0,\alpha} \nr{Dw}_{L^{p}(B_{r})}^{q-p} \left( \mint_{B_r} \snr{Dw}^{q_*} dx \right)^{\frac{p}{q_*}},
\end{flalign}
\cd{$c=c(n,q)$} and
\begin{flalign}\label{poincare_s}
\mint_{B_r} b(x) \frac{\snr{w-(w)_r}^s}{r^s} \, dx & \leq c[b]_{0,\beta}r^{\beta} \left( \mint_{B_r} \snr{Dw}^{p} dx \right)^{\frac{s-p}{p}} \left( \mint_{B_r} \snr{Dw}^{s_*} dx \right)^{\frac{p}{s_*}} \nonumber \\
& \leq c[b]_{0,\beta}r^{\beta-\frac{n(s-p)}{p}} \nr{Dw}_{L^{p}(B_{r})}^{s-p} \left( \mint_{B_r} \snr{Dw}^{s_*} dx \right)^{\frac{p}{s_*}} \nonumber \\
& \leq c[b]_{0,\beta} \nr{Dw}_{L^{p}(B_{r})}^{s-p} \left( \mint_{B_r} \snr{Dw}^{s_*} dx \right)^{\frac{p}{s_*}},
\end{flalign}
\cd{with $c=c(n,s)$}. In addition, it is clear that
\begin{flalign*}
\mint_{B_r} \frac{\snr{w-(w)_r}^p}{r^p} \, dx \leq \cd{c}\left( \mint_{B_r} \snr{Dw}^{p_*} dx \right)^{\frac{p}{p_*}},
\end{flalign*}
where \cd{$c=c(n,p)$ and} $p_* := \max \left\lbrace \frac{np}{n+p},1 \right\rbrace$.
We remark from \eqref{hmvp} that $p_* \cd{<} q_* \cd{<} s_*$.
Combining these estimates, we get
\begin{flalign}\label{result_poincare_case_0}
\mint_{B_{r}}H\left(x, \frac{w-(w)_r}{r} \right) dx & \leq c\left(1+[a]_{0,\alpha}\nr{Dw}^{q-p}_{L^{p}(B_{r})}+[b]_{0,\beta}\nr{Dw}^{s-p}_{L^{p}(B_{r})}\right) \left( \mint_{B_r} \snr{Dw}^{s_*} dx \right)^{\frac{p}{s_*}} \nonumber \\
& \leq c\left(1+[a]_{0,\alpha}\nr{Dw}^{q-p}_{L^{p}(B_{r})}+[b]_{0,\beta}\nr{Dw}^{s-p}_{L^{p}(B_{r})}\right)\left(\mint_{B_{r}}H(x,Dw)^{d_0} \, dx\right)^{\frac{1}{d_0}},
\end{flalign}
where $d_0 := s_*/p \in (0,1)$ \cd{and $c=c(n,p,q,s)$}. \\
We now turn to the case
\begin{equation}\label{poincare_case_1}
\sup_{x \in B_r} a(x) > 4[a]_{0,\alpha}r^{\alpha} \quad \mbox{and} \quad \sup_{x \in B_r} b(x) \le 4[b]_{0,\beta}r^{\beta}.
\end{equation}
Then there exists a point $y_0 \in B_r$ such that $a_0 := a(y_0) > 4[a]_{0,\alpha}r^{\alpha}$.
This gives
\begin{flalign*}
\snr{a(x)-a_0} \leq [a]_{0,\alpha}(2r)^{\alpha} \leq 2[a]_{0,\alpha}r^{\alpha} < \frac{a_0}{2}, \quad \forall x \in B_r,
\end{flalign*}
and hence $a(x) \leq 2a_0$ and $a_0 \leq 2a(x)$.
Therefore, we have
\begin{flalign*}
\frac{1}{2} \left( |t|^p + a_0 |t|^q \right) \leq |t|^p + a(x)|t|^q \leq |t|^p + a_0 |t|^q, \quad \forall x \in B_r, \ t \in \mathbb{R}.
\end{flalign*}
This and \eqref{poincare_s} yield
\begin{flalign*}
\mint_{B_{r}}H\left(x, \frac{w-(w)_r}{r} \right) dx & \leq 2 \mint_{B_r} H_{0}^{q}\left(\frac{w-(w)_r}{r} \right) dx + c[b]_{0,\beta} \nr{Dw}_{L^{p}(B_{r})}^{s-p} \left( \mint_{B_r} \snr{Dw}^{s_*} dx \right)^{\frac{p}{s_*}},
\end{flalign*}
\cd{with $c=c(n,s)$}. Using Sobolev-Poincar\'e inequality for Young function $H_{0}^{q}$, we have
\begin{flalign}\label{result_poincare_case_1}
\mint_{B_{r}}H\left(x, \frac{w-(w)_r}{r} \right) dx & \leq c \left( \mint_{B_r} H_{0}^{q}(Dw)^{d_q} \, dx \right)^{\frac{1}{d_q}} + c[b]_{0,\beta} \nr{Dw}_{L^{p}(B_{r})}^{s-p} \left( \mint_{B_r} \snr{Dw}^{s_*} dx \right)^{\frac{p}{s_*}} \nonumber \\
& \leq c \left( \mint_{B_r} H(x,Dw)^{d_q} \, dx \right)^{\frac{1}{d_q}} + c[b]_{0,\beta} \nr{Dw}_{L^{p}(B_{r})}^{s-p} \left( \mint_{B_r} H(x,Dw)^{\frac{s_*}{p}} dx \right)^{\frac{p}{s_*}} \nonumber \\
& \leq c\left(1+[b]_{0,\beta}\nr{Dw}^{s-p}_{L^{p}(B_{r})}\right)\left(\mint_{B_{r}}H(x,Dw)^{d_1} \, dx\right)^{\frac{1}{d_1}},
\end{flalign}
where \cd{$c=c(n,p,q,s)$}, $d_q = d_q(n,p,q) \in (0,1)$ and $d_1 := \max \left\lbrace d_q, \frac{s_*}{p} \right\rbrace \in (0,1)$. \\
As in the case \eqref{poincare_case_1}, we can obtain the estimate
\begin{flalign}\label{result_poincare_case_2}
\mint_{B_{r}}H\left(x, \frac{w-(w)_r}{r} \right) dx \leq c\left(1+[a]_{0,\alpha}\nr{Dw}^{q-p}_{L^{p}(B_{r})}\right)\left(\mint_{B_{r}}H(x,Dw)^{d_2} \, dx\right)^{\frac{1}{d_2}},
\end{flalign}
for \cd{$c=c(n,p,q,s)$} and some $d_2\cd{=d_{2}(n,p,q,s)} \in (0,1)$, for the case
\begin{equation}\label{poincare_case_2}
\sup_{x \in B_r} a(x) \le 4[a]_{0,\alpha}r^{\alpha} \quad \mbox{and} \quad \sup_{x \in B_r} b(x) > 4[b]_{0,\beta}r^{\beta}.
\end{equation}
Finally, let us consider the case
\begin{equation}\label{poincare_case_3}
\sup_{x \in B_r} a(x) > 4[a]_{0,\alpha}r^{\alpha} \quad \mbox{and} \quad \sup_{x \in B_r} b(x) > 4[b]_{0,\beta}r^{\beta}.
\end{equation}
We see that there exist points $y_0, z_0 \in B_r$ such that $a_0 := a(y_0) > 4[a]_{0,\alpha}r^{\alpha}$ and $b_0 := b(z_0) > 4[b]_{0,\beta}r^{\beta}$.
This yields
\begin{flalign*}
\snr{a(x)-a_0} \leq [a]_{0,\alpha}(2r)^{\alpha} \leq 2[a]_{0,\alpha}r^{\alpha} < \frac{a_0}{2}, \quad \forall x \in B_r,
\end{flalign*}
and
\begin{flalign*}
\snr{b(x)-b_0} \leq [b]_{0,\beta}(2r)^{\beta} \leq 2[b]_{0,\beta}r^{\beta} < \frac{b_0}{2}, \quad \forall x \in B_r.
\end{flalign*}
It follows that $\frac{1}{2}a_0 \leq a(x) \leq 2a_0$ and $\frac{1}{2}b_0 \leq b(x) \leq 2b_0$, and hence
\begin{flalign*}
\frac{1}{2} H_{0}(t) \leq H(x,t) \leq 2H_{0}(t), \quad \forall x \in B_r, \ t \in \mathbb{R}.
\end{flalign*}
We now use Sobolev-Poincar\'e inequality for Young function $H_{0}$ to obtain
\begin{flalign}\label{result_poincare_case_3}
\mint_{B_{r}}H\left(x, \frac{w-(w)_r}{r} \right) dx & \leq 2 \mint_{B_{r}}H_{0}\left(\frac{w-(w)_r}{r} \right) dx \nonumber \\
& \leq c \left( \mint_{B_{r}}H_{0}(Dw)^{d_3} dx \right)^{\frac{1}{d_3}} \leq c \left( \mint_{B_{r}}H(x,Dw)^{d_3} dx \right)^{\frac{1}{d_3}}
\end{flalign}
for \cd{$c=c(n,p,q,s)$ and} some $d_3=d_3(n,p,q,s) \in (0,1)$. \\
Setting $d := \max \{ d_0, d_1, d_2, d_3 \} \in (0,1)$, we conclude from \eqref{result_poincare_case_0}, \eqref{result_poincare_case_1}, \eqref{result_poincare_case_2} and \eqref{result_poincare_case_3} that
\begin{flalign*}
\mint_{B_{r}}H\left(x, \frac{w-(w)_r}{r} \right) dx \le c\left(1+[a]_{0,\alpha}\nr{Dw}^{q-p}_{L^{p}(B_{r})}+[b]_{0,\beta}\nr{Dw}^{s-p}_{L^{p}(B_{r})}\right)\left(\mint_{B_{r}}H(x,Dw)^{d} \, dx\right)^{\frac{1}{d}},
\end{flalign*}
which completes the proof.
\end{proof}
}
\cd{\begin{remark}\label{poi}
\emph{An inequality of the type of \eqref{sp} holds for general Sobolev maps $w\in W^{1,H(\cdot)}$ such that $w\equiv 0$ on a set $A$ such that $\snr{A}\ge \gamma\snr{B_{r}}$. Precisely, we have that
\begin{flalign}\label{poi0}
\mint_{B_{r}}H\left(x,\frac{w}{r}\right) \, dx \le c\left(\mint_{B_{r}}H(x,Dw)^{d} \ dx\right)^{\frac{1}{d}},
\end{flalign}
where $d<1$ is the same as the one appearing in \eqref{sp} and $c=c(\gamma,n,p,q,s,[a]_{0,\alpha},[b]_{0,\beta},\alpha,\beta,\nr{Dw}_{L^{p}(B_{r})})$. 
}
\end{remark}}

\begin{lemma}[Caccioppoli Inequalities]\label{L2}
Let $u \in W^{1,H(\cdot)}_{\mathrm{loc}}(\Omega)$ be a local minimizer of \eqref{hmvp}, with \oh{$a(\cdot)$, $b(\cdot)$} and $p$, $q$, $s$ satisfy \eqref{mathAB} and \eqref{rat} respectively. Then there exists \oh{a constant $c=c(n,p,q,s)>0$} such that
\begin{flalign}\label{cacc}
\mint_{B_{\rho}}H(x,\oh{D}u) \ \oh{dx} \le c \mint_{B_{r}}H\left(x, \frac{u-\oh{(u)_r}}{r-\rho}\right) \ \oh{dx},
\end{flalign}
and for $\kappa \in \mathbb{R}$,
\begin{flalign}\label{cacck}
\int_{B_{\rho}}H(x,\oh{D}(u-\kappa)_{\pm}) \ \oh{dx} \le c\int_{B_{R}}H\left(x, \frac{(u-\kappa)_{\pm}}{R-\rho}\right) \ \oh{dx}.
\end{flalign}
\end{lemma}

A direct consequence of \eqref{cacc} is the following inner local higher integrability result of Gehring type.
\begin{lemma}[Gehring's Lemma]\label{L3}
There are $c=c(n,p,q,s,[a]_{0,\alpha},[b]_{0,\beta}, \nr{Du}_{L^{p}(B_{r})})>0$ and a positive integrability exponent $\delta_{g}=\delta_{g}(n,p,q,s,[a]_{0,\alpha},[b]_{0,\beta}, \nr{Du}_{L^{p}(B_{r})})$ such that if $u \in W^{1,p}_{\mathrm{loc}}(\Omega)$ is a local minimizer, then
\begin{flalign}\label{geh}
H(\cdot,\oh{D}u) \in L^{1+\delta_{g}}_{\mathrm{loc}}(\Omega) \quad \mathrm{and} \quad \left(\mint_{B_{\oh{r/2}}}H(x,\oh{D}u)^{1+\delta_{g}}\ \oh{dx} \right)^{\frac{1}{1+\delta_{g}}}\le c\mint_{B_{r}}H(x,\oh{D}u) \ \oh{dx}, \quad \oh{\forall B_r \subset \Omega.}
\end{flalign}
\end{lemma}
After a standard covering argument, \cd{it follows from Lemma \ref{L3} that $u \in W^{1,p(1+\delta_{g})}_{\mathrm{loc}}(\Omega)$, so $u \in W^{1,p(1+\delta_{g})}(\Omega_{0})$} for $\Omega_{0}\Subset \Omega$. \cd{Moreover, by H\"older inequality, \eqref{geh} is true if $\delta_{g}$ is replaced by any $\sigma \in (0,\delta_{g})$.}

The next one is an up to the boundary higher integrability result for a solution of Dirichlet problems related to the multi-phase energy $H$. Clearly, when \oh{$a(\cdot) \equiv a_{0} = \const$} and \oh{$b(\cdot) \equiv b_{0} = \const$}, it extends to the auxiliary Young functions $H_{0}^{p}$, $H_{0}^{q}$, $H_{0}^{s}$ and $H_{0}$. \cd{In this case, $[a]_{0,\alpha}=[a_{0}]_{0,\alpha}=0$ and $[b]_{0,\beta}=[b_{0}]_{0,\beta}=0$, so constants and exponents do not depend either on $[a]_{0,\alpha}$, $[b]_{0,\beta}$ nor on $\nr{Dv}_{L^{p}(B_{r})}$}.

\begin{lemma}[Higher integrability up to the boundary]\label{L5}
Let $B_{r}\Subset \Omega_{0}\Subset \Omega$, $1<p\le q\le s$ and $v \in W^{1,H(\cdot)}_{u}(B_{r})$ be a solution to the Dirichlet problem
\begin{flalign}\label{dphm}
v \mapsto \min_{w\in W^{1,H(\cdot)}_{u}(B_{r})}\int_{B_{r}}H(x,Dw) \ dx,
\end{flalign}
and $\delta_{0}>0$ be such that $u \in \oh{W^{1,H(\cdot)^{1+\delta_{0}}}(B_{r})}$. Then there \oh{exists} $0<\sigma_{g}<\delta_{0}$, so that $v \in \oh{W^{1,H(\cdot)^{1+\sigma_{g}}}(B_{r})}$ and
\begin{flalign}\label{63}
\mint_{B_{r}} \oh{H(x,Dv)^{1+\sigma_{g}}} \ dx \le c\left\{\left(\mint_{B_{r}}H(x,Dv) \ dx\right)^{1+\sigma_{g}}+\mint_{B_{r}}\oh{H(x,Du)^{1+\sigma_{g}}} \ dx\right\},
\end{flalign}
where $c=c(n,p,q,s,[a]_{0,\alpha},[b]_{0,\beta}, \nr{H(\cdot,Du)}_{L^{1}(B_{r})})$ and $\oh{\sigma_{g}}=\oh{\sigma_{g}}(n,p,q,s,[a]_{0,\alpha},[b]_{0,\beta}, \nr{H(\cdot,Du)}_{L^{1}(B_{r})})$. 
\end{lemma}
\cd{\begin{proof}
With $x_0\in B_r$, let us fix a ball $B_{\rho}(x_{0})\subset \mathbb{R}^{n}$. We start with the case in which it is $\snr{B_{\rho}(x_{0})\setminus B_{r}}>\frac{\snr{B_{\rho}(x_{0})}}{10}$. Let us fix $\rho/2<t<s<\rho$ and take a cut-off function 
$
\eta \in C^{1}_{c}(B_{s}(x_{0}))$ such that 
$\chi_{B_{t}(x_{0})}\le \eta \le \chi_{B_{s}(x_{0})}$ and $\snr{D\eta}\le 2/(s-t)$.
Since $\left.v\right|_{\partial B_{r}}=\left.u\right|_{\partial B_{r}}$ and $\left.\eta\right|_{\partial B_{s}(x_{0})}=0$, the function $v-\eta(v-u)$ coincides with $v$ on $\partial B_r$ and on $\partial B_{s}(x_{0})$ in the sense of traces and therefore, by the minimality of $v$ and the features of $\eta$ we obtain
\begin{flalign*}
\int_{B_{s}(x_{0})\cap B_{r}}H(x,Dv) \ dx \le c\left\{\int_{(B_{s}(x_{0})\setminus B_{t}(x_{0}))\cap B_{r}}H(x,Dv) \ dx +\int_{B_{s}(x_{0})\cap B_{r}}H(x,Du)+H\left(x,\frac{v-u}{r}\right) \ dx\right\},
\end{flalign*}
with $c=c(n,p,q,s)$.\\
By the classical hole-filling technique and Lemma \ref{iter}, we can conclude that
\begin{flalign}\label{130}
\int_{B_{\rho/2}(x_{0})\cap B_{r}}H(x,Dv) \ dx \le c\int_{B_{\rho}(x_{0})\cap B_{r}}H(x,Du)+H\left(x,\frac{v-u}{r}\right) \ dx,
\end{flalign}
for $c=c(n,p,q,s)$. Now extend $v-u$ as zero outside $B_{r}$ and recall that $\snr{B_{\rho}(x_{0})}\ge \snr{B_{\rho}(x_{0})\setminus B_{r}}> \frac{\snr{B_{\rho}(x_{0})}}{10}$. Poincar\'e's inequality \eqref{poi0} applies, thus getting
\begin{flalign}\label{131}
\mint_{B_{\rho}(x_{0})\cap B_{r}}H\left(x,\frac{v-u}{r}\right) \ dx \le c\left\{\left(\mint_{B_{\rho}\cap B_{r}}H(x,Dv)^{d} \ dx\right)^{\frac{1}{d}}+\mint_{B_{\rho}(x_{0})\cap B_{r}}H(x,Du) \ dx \right\},
\end{flalign}
with $c=c(n,p,q,s,[a]_{0,\alpha},[b]_{0,\beta},\alpha,\beta,\nr{H(\cdot,Du)}_{L^{1}(B_{r})})$. Here we dispensed $c$ from the dependence of $\nr{Dv}_{L^{p}(B_{r})}$ by using the minimality of $v$ and the fact that $\left.v\right |_{\partial B_{r}}=\left.u\right |_{\partial B_{r}}$. Merging \eqref{130} and \eqref{131} we obtain
\begin{flalign*}
\mint_{B_{\rho/2}(x_{0})\cap B_{r}}H(x,Dv) \ dx \le c\left\{\left(\mint_{B_{\rho}\cap B_{r}}H(x,Dv)^{d} \ dx\right)^{\frac{1}{d}}+\mint_{B_{\rho}(x_{0})\cap B_{r}}H(x,Du) \ dx \right\}.
\end{flalign*}
We next consider the situation when it is $B_\rho(x_{0}) \Subset B_r$, in which case the proof is analogous to the one for the interior case. As mentioned in Remark \ref{poi}, we can assume that the exponent $d<1$ from \rif{sp} and \rif{poi0} is the same. The two cases can be combined via a standard covering argument. In fact, let us define 
\begin{flalign*}
V(x)=\begin{cases}
\ H(x,Dv(x))^{d} \quad & \ \mathrm{in} \ B_{r}\\
\ 0 \quad & \ \mathrm{in} \ \mathbb{R}^{n}\setminus B_{r}
\end{cases}\quad and \quad U(x)=\begin{cases}
\ H(x,Du(x)) \quad & \ \mathrm{in} \ B_{r}\\
\ 0 \quad & \ \mathrm{in} \ \mathbb{R}^{n}\setminus B_{r}
\end{cases},
\end{flalign*}
we get
\begin{flalign*}
\mint_{B_{\rho/2}(x_{0})}V(x)^{\frac{1}{d}} \, dx\le c\left \{ \left(\mint_{B_{\rho}(x_{0})}V(x) \, dx\right)^{\frac{1}{d}}+\mint_{B_{\rho}(x_{0})}U(x) \, dx \right\},
\end{flalign*}
with $c=c=c(n,p,q,s,[a]_{0,\alpha},[b]_{0,\beta},\alpha,\beta,\nr{H(\cdot,Du)}_{L^{1}(B_{r})})$ and $0<d<1$. At this point the conclusion follows by a standard variant of Gehring's lemma. 
\end{proof}}

Furthermore, $u$ is locally bounded.
\begin{lemma}\label{L4}
Let $u\in W^{1,H_{M}(\cdot)}_{\mathrm{loc}}(\Omega)$ be a local minimizer of \eqref{hmvp}. Then $u$ is locally bounded \oh{in} $\Omega$ and for any $\Omega_{0}\Subset \Omega$ there is a \oh{positive} constant $c=c(\texttt{data}(\Omega_{0}))$ such that $\nr{u}_{L^{\infty}(\Omega_{0})}\le c$.
\end{lemma}
\begin{proof}
This can be obtained as in \cite{double}, Section $10$ as a consequence of \eqref{cacck} or by noticing that the generalized Young function $H(x,t)=t^{p}+a(x)t^{q}+b(x)t^{s}$ \oh{under the assumptions \eqref{mathAB} and \eqref{rat}} satisfies hypotheses (A0), (A1), (AInc) and (ADec) of Theorem $1.3$ in \cite{hasto}. In fact, with the notation used in \cite{hasto}, it is easy to see that $H^{+}(\delta)\le 1\le H^{-}(1)$ for $\delta=\left[\frac{1}{3}\left(1+\max\left\{\nr{a}_{\infty},\nr{b}_{\infty}\right\}\right)^{-1}\right]^{\frac{1}{p}}\in (0,1)$. (A1) is true by choosing $\gamma=\frac{1}{2}\min\left\{ \left(\frac{1}{3}\right)^{\frac{1}{p}},\left(\frac{\omega_{n}^{\frac{1}{p}}}{3[a]_{0,\alpha}\diam(\Omega)^{\alpha-n\frac{q-p}{p}}}\right)^{\frac{1}{q-p}},\left(\frac{\omega_{n}^{\frac{1}{p}}}{3[b]_{0,\beta}\diam(\Omega)^{\beta-n\frac{s-p}{p}}}\right)^{\frac{1}{s-p}} \right\} \oh{\in (0,1)}$\oh{, where $\omega_{n}$ is the volume of the unit ball $B_1 \subset \mathbb{R}^n$}. (AInc) clearly holds with $\gamma^{-}=p>1$ and (ADec) is verified by $\gamma^{+}=s\ge p>1$.
\end{proof}

\section{Different alternatives}
For later uses, we also define the quantities
\begin{flalign}\label{infab}
a_{i}(B_{r})=\inf_{x \in B_{r}}a(x) \quad \mathrm{and}\quad b_{i}(B_{r})=\inf_{x\in B_{r}}b(x),
\end{flalign}
which will play an important role along the proof. In fact, when dealing with those so called non uniformly elliptic problems, the question of the degeneracy of the coefficients is crucial. Precisely we will look at four different scenarios:
\begin{flalign*}
\begin{cases}
\ \texttt{deg}(B_{r})\colon \ \  a_{i}(B_{r})\le 4[a]_{0,\alpha}r^{\alpha-\gamma_{a}} \quad &\mathrm{and} \quad b_{i}(B_{r})\le 4[b]_{0,\beta}r^{\beta-\gamma_{\beta}}\\
\ \texttt{deg}_{\alpha}(B_{r})\colon \ \  a_{i}(B_{r})\le 4[a]_{0,\alpha}r^{\alpha-\gamma_{a}} \quad &\mathrm{and} \quad b_{i}(B_{r})> 4[b]_{0,\beta}r^{\beta-\gamma_{b}}\\
\ \texttt{deg}_{\beta}(B_{r})\colon \ \  a_{i}(B_{r})> 4[a]_{0,\alpha}r^{\alpha-\gamma_{a}} \quad &\mathrm{and} \quad b_{i}(B_{r})\le 4[b]_{0,\beta}r^{\beta-\gamma_{b}}\\
\ \texttt{ndeg}(B_{r})\colon \ \  a_{i}(B_{r})\ge 4[a]_{0,\alpha}r^{\alpha-\gamma_{a}} \quad &\mathrm{and} \quad b_{i}(B_{r}) >4[b]_{0,\beta}r^{\beta-\gamma_{b}}\,,
\end{cases}
\end{flalign*}
where
\oh{
\begin{flalign}\label{ga}
\gamma_{a}=\begin{cases}
\ 0 \quad & \mbox{if} \quad n \ge p(1+\delta_{g})\\
\ \alpha-\frac{n(q-p)}{p}+\frac{n\delta_{g}(q-p)}{2p(1+\delta_{g})} \quad & \mbox{if} \quad n<p(1+\delta_{g})
\end{cases}
\end{flalign}
and
\begin{flalign}\label{gb}
\gamma_{b}=\begin{cases}
\ 0 \quad & \mbox{if} \quad n \ge p(1+\delta_{g})\\
\ \beta-\frac{n(s-p)}{p}+\frac{n\delta_{g}(s-p)}{2p(1+\delta_{g})} \quad & \mbox{if} \quad n<p(1+\delta_{g})
\end{cases},
\end{flalign}}
where $\delta_{g}$ is the higher integrability exponent given by Gehring Lemma which can be found in Section \ref{3}.\\ 
The above four cases, suitably combined, will render the desired regularity. To shorten the notation, we shall summarize the dependencies from the characteristics of the integrand we are dealing with, as
\oh{
\begin{flalign*}
\texttt{data}(\Omega_{0}) \equiv \begin{cases} \left( n,p,q,s,[a]_{0,\alpha},[b]_{0,\beta},\alpha, \beta, \nr{u}_{L^{\infty}(\Omega_{0})},  \nr{H(\cdot,Du)}_{L^{1+\delta_{g}}(\Omega_{0})} \right) \quad & \mbox{if} \quad n \ge p(1+\delta_{g})\\
\left( n,p,q,s,[a]_{0,\alpha},[b]_{0,\beta},\alpha, \beta, [u]_{C^{0,\lambda_{g}}(\Omega_{0})},  \nr{H(\cdot,Du)}_{L^{1+\delta_{g}}(\Omega_{0})} \right) \quad & \mbox{if} \quad n<p(1+\delta_{g})
\end{cases},
\end{flalign*}
and 
\begin{flalign*}
\texttt{data} \equiv \left( n,p,q,s,\nr{a}_{L^{\infty}(\Omega)}, \nr{b}_{L^{\infty}(\Omega)}, [a]_{0,\alpha},[b]_{0,\beta} \right).
\end{flalign*}}
\oh{Here, $\lambda_{g}=1-\frac{n}{p(1+\delta_{g})}$ is the H\"older continuity exponent coming from Sobolev-Morrey's embedding theorem when $n<p(1+\delta_{g})$ and} $\Omega_{0}\Subset \Omega$ is any open set compactly contained in $\Omega$. This will be helpful, since all the existing results we are going to use are of local nature.\\

Exploiting the different phases $(\texttt{deg})$-$(\texttt{ndeg})$ we obtain various forms of the previous Caccioppoli's inequality. We collect them in the next Corollary. Moreover, the constants $a_{0}$ and $b_{0}$ appearing in the definition of the auxiliary Young functions $H_{0}^{p}$, $H^{q}_{0}$, $H^{s}_{0}$ and $H_{0}$ will take the values $a_{0}=a_{i}(B_{2r})$ and $b_{0}=b_{i}(B_{2r})$.

\begin{corollary}\label{C0}
Let $u \in W^{1,H(\cdot)}_{\mathrm{loc}}(\Omega)$ be a local minimizer of \eqref{hmvp} and $B_{r}$, $r \in (0,1)$ be any ball such that $B_{2r}\Subset \Omega_{0}\Subset \Omega$. Then the following is verified:\\
\begin{flalign}
\texttt{deg}(B_{2r})\Rightarrow&  \mint_{B_{r}}H(x,Du) \ dx \le c_{1}\mint_{B_{2r}}H_{0}^{p}\left(\frac{u-(u)_{2r}}{2r}\right) \ dx, \label{degc}\\
\texttt{deg}_{\alpha}(B_{2r})\Rightarrow& \mint_{B_{r}}H(x,Du) \ dx \le c_{2} \mint_{B_{2r}}H_{0}^{s}\left(\frac{u-(u)_{2r}}{2r}\right) \ dx,\label{degac}\\
\texttt{deg}_{\beta}(B_{2r})\Rightarrow& \mint_{B_{r}}H(x,Du) \ dx \le c_{3}\mint_{B_{2r}}H_{0}^{q}\left(\frac{u-(u)_{2r}}{2r}\right) \ dx, \label{degbc}\\
\texttt{ndeg}(B_{2r})\Rightarrow& \mint_{B_{r}}H(x,Du) \ dx \le c_{4}\mint_{B_{2r}}H_{0}\left(\frac{u-(u)_{2r}}{2r}\right) \ dx,\label{ndegc}.
\end{flalign}
Here, if $n\ge (1+\delta_{g})p$,  $c_{1}=c_{1}(n,p,q,s,[a]_{0,\alpha},[b]_{0,\beta},\alpha,\beta,\nr{u}_{L^{\infty}(\Omega_{0})})$, $c_{2}=c_{2}(n,p,s,q,[a]_{0,\alpha},\alpha, \beta, \nr{u}_{L^{\infty}(\Omega_{0})})$,\\ $c_{3}=c_{3}(n,p,q,s,[b]_{0,\beta},\alpha, \beta, \nr{u}_{L^{\infty}(\Omega_{0})})$ and $c_{4}=c_{4}(n,p,q,s,[a]_{0,\alpha},[b]_{0,\beta},\alpha,\beta)$, while, if $n<p(1+\delta_{g})$, $c_{1}=c_{1}(n,p,q,s,[a]_{0,\alpha},[b]_{0,\beta},\alpha,\beta,[u]_{C^{0,\lambda_{g}}(\Omega_{0})})$, $c_{2}=c_{2}(n,p,s,q,[a]_{0,\alpha},\alpha, \beta, [u]_{C^{0,\lambda_{g}}(\Omega_{0})})$,\\ $c_{3}=c_{3}(n,p,q,s,[b]_{0,\beta},\alpha, \beta, [u]_{C^{0,\lambda_{g}}(\Omega_{0})})$ and $c_{4}=c_{4}(n,p,q,s,[a]_{0,\alpha},[b]_{0,\beta},\alpha,\beta,[u]_{C^{0,\lambda_{g}}(\Omega_{0})})$.

\end{corollary}

\begin{proof}
First, notice that, by \eqref{rat}, $\gamma_{a}\ge0$ and $\gamma_{b}\ge 0$. Moreover, if $n\ge p(1+\delta_{g})$ we see that
\oh{
\begin{flalign}\label{gam}
&\alpha-\gamma_{a}+p-q \ge \frac{n(q-p)}{p}-(q-p) \ge \delta_{g}(q-p) > 0,\\
&\beta-\gamma_{b}+p-s \ge \frac{n(s-p)}{p}-(s-p) \ge \delta_{g}(s-p) > 0,\label{gbm}
\end{flalign}
while, if $n< p(1+\delta_{g})$,
\begin{flalign}\label{gam1}
&\alpha-\gamma_{a}+(\lambda_{g}-1)(q-p) = \frac{n\delta_{g}(q-p)}{2p(1+\delta_{g})} > 0 ,\\
&\beta-\gamma_{b}+(\lambda_{g}-1)(s-p)=\frac{n\delta_{g}(s-p)}{2p(1+\delta_{g})}>0.\label{gam2}
\end{flalign}}
Assume $\texttt{deg}(B_{2r})$.
\oh{We observe that for any $x \in B_{2r}$,
\begin{flalign*}
a(x) &= \left(a(x)-a_{i}(B_{2r})\right)+a_{i}(B_{2r}) \\
&\le [a]_{0,\alpha}(4r)^{\alpha} + 4[a]_{0,\alpha}r^{\alpha-\gamma_{a}} \le 8[a]_{0,\alpha}r^{\alpha-\gamma_{a}}
\end{flalign*}
since $\gamma_{a} \ge 0$ and $r \in (0,1)$.
Similarly we have $b(x) \le 8[b]_{0,\beta}r^{\beta-\gamma_{b}}, \forall x \in B_{2r}$.
}
If $n\ge p(1+\delta_{g})$, from \eqref{cacc}, Lemma \ref{L4}, \eqref{gam} and \eqref{gbm} we get,
\begin{flalign*}
\mint_{B_{r}}H(x,Du) \ dx \le& c\mint_{B_{2r}}H\left(x,\frac{u-(u)_{2r}}{\oh{r}}\right) \ dx\\
\le &c\mint_{B_{2r}}\left(1+\oh{8}[a]_{0,\alpha}r^{\alpha-\gamma_{a}+p-q}\nr{u}_{L^{\infty}(\Omega_{0})}^{q-p}+\oh{8}[b]_{0,\beta}r^{\beta-\gamma_{b}+p-s}\nr{u}_{L^{\infty}(\Omega_{0})}^{s-p}\right)\left |\frac{u-(u)_{2r}}{2r} \right |^{p} \ dx\\
\le & c_{1}\mint_{B_{2r}}H_{0}^{p}\left(\frac{u-(u)_{2r}}{2r}\right) \ dx,
\end{flalign*}
where $c_{1}=c_{1}(n,p,q,s,[a]_{0,\alpha},[b]_{0,\beta},\alpha,\beta,\nr{u}_{L^{\infty}(\Omega_{0})})$. On the other hand, if $n<p(1+\delta_{g})$ proceding as before but using Sobolev-Morrey's theorem and \eqref{gam1}, \eqref{gam2} instead of \eqref{gam}, \eqref{gbm}, we obtain
\begin{flalign*}
\mint_{B_{r}}H(x,Du) \ dx \le& c\mint_{B_{2r}}H\left(x,\frac{u-(u)_{2r}}{\oh{r}}\right) \ dx\\
\le &c\mint_{B_{2r}}\left(1+\oh{8}[a]_{0,\alpha}r^{\alpha-\gamma_{a}+(\lambda_{g}-1)(q-p)}[u]_{C^{0,\lambda_{g}}(\Omega_{0})}^{q-p}\right.\\
& \qquad \quad +\left.\oh{8}[b]_{0,\beta}r^{\beta-\gamma_{b}+(\lambda_{g}-1)(s-p)}[u]_{C^{0,\lambda_{g}}(\Omega_{0})}^{s-p}\right)\left |\frac{u-(u)_{2r}}{2r} \right |^{p} \ dx\\
\le & c_{1}\mint_{B_{2r}}H_{0}^{p}\left(\frac{u-(u)_{2r}}{2r}\right) \ dx,
\end{flalign*}
where $c_{1}=c_{1}(n,p,q,s,[a]_{0,\alpha},[b]_{0,\beta},\alpha,\beta,[u]_{C^{0,\lambda_{g}}(\Omega_{0})})$.\\
Now suppose $\texttt{deg}_{\alpha}(B_{2r})$. If $n\ge p(1+\delta_{g})$, \oh{we see from \eqref{cacc}, \eqref{gam}, \eqref{gbm} and Lemma \ref{L4} that}
\begin{flalign*}
\mint_{B_{r}}H(x,Du) \ dx \le& c\mint_{B_{2r}}H\left(x,\frac{u-(u)_{2r}}{r}\right) \ dx\\
\le &c \mint_{B_{2r}}\left(1+\oh{8}[a]_{0,\alpha}r^{\alpha-\gamma_{a}+p-q}\nr{u}_{L^{\infty}(\Omega_{0})}^{q-p}\right)\left |\frac{u-(u)_{2r}}{2r} \right |^{p} \ dx\\
&+c\mint_{B_{2r}}\left(b(x)-b_{i}(B_{2r})\right)\left |\frac{u-(u)_{2r}}{2r} \right|^{s}+b_{i}(B_{2r})\left |\frac{u-(u)_{2r}}{2r} \right|^{s} \ dx\\
\le &c\mint_{B_{2r}}\left |\frac{u-(u)_{2r}}{2r} \right | ^{p}+[b]_{0,\beta}(\oh{4r})^{\beta}\left |\frac{u-(u)_{2r}}{2r} \right|^{s}+b_{i}(B_{2r})\left |\frac{u-(u)_{2r}}{2r} \right|^{s} \ dx\\ 
\le &c\mint_{B_{2r}}\left |\frac{u-(u)_{2r}}{2r} \right |^{p}+\oh{2b_{i}(B_{2r})}\left |\frac{u-(u)_{2r}}{2r} \right |^{s} \ dx\le c_{2}\mint_{B_{2r}}H_{0}^{s}\left(\frac{u-(u)_{2r}}{\oh{2r}}\right) \ dx,
\end{flalign*}
since, being $r\in (0,1)$, $r^{\beta}\le r^{\beta-\gamma_{b}}$. Here, $c_{2}=c_{2}(n,p,q,s,[a]_{0,\alpha},\alpha,\nr{u}_{L^{\infty}(\Omega_{0})})$. If $n<p(1+\delta_{g})$ we have, by exploiting \eqref{gam1} and \eqref{gam2},
\begin{flalign*}
\mint_{B_{r}}H(x,Du) \ dx \le& c\mint_{B_{2r}}H\left(x,\frac{u-(u)_{2r}}{r}\right) \ dx\\
\le &c \mint_{B_{2r}}\left(1+\oh{8}[a]_{0,\alpha}r^{\alpha-\gamma_{a}+(\lambda_{g}-1)(q-p)}[u]_{C^{0,\lambda_{g}}(\Omega_{0})}^{q-p}\right)\left |\frac{u-(u)_{2r}}{2r} \right |^{p} \ dx\\
&+c\mint_{B_{2r}}\left(b(x)-b_{i}(B_{2r})\right)\left |\frac{u-(u)_{2r}}{2r} \right|^{s}+b_{i}(B_{2r})\left |\frac{u-(u)_{2r}}{2r} \right|^{s} \ dx\\
\le &c\mint_{B_{2r}}\left |\frac{u-(u)_{2r}}{2r} \right | ^{p}+[b]_{0,\beta}(\oh{4r})^{\beta}\left |\frac{u-(u)_{2r}}{2r} \right|^{s}+b_{i}(B_{2r})\left |\frac{u-(u)_{2r}}{2r} \right|^{s} \ dx\\ 
\le &c\mint_{B_{2r}}\left |\frac{u-(u)_{2r}}{2r} \right |^{p}+\oh{2b_{i}(B_{2r})}\left |\frac{u-(u)_{2r}}{2r} \right |^{s} \ dx\le c_{2}\mint_{B_{2r}}H_{0}^{s}\left(\frac{u-(u)_{2r}}{\oh{2r}}\right) \ dx,
\end{flalign*}
with \oh{$c_{2}=c_{2}(n,p,q,s,[a]_{0,\alpha},\alpha,[u]_{C^{0,\lambda_{g}}(\Omega_{0})})$}.
If $\texttt{deg}_{\beta}(B_{2r})$ is in force, then, as before, for $n\ge p(1+\delta_{g})$, we have
\begin{flalign*}
\mint_{B_{r}}H(x,Du) \ dx \le& c\mint_{B_{2r}}H\left(x,\frac{u-(u)_{2r}}{r}\right) \ dx\\
\le &c \mint_{B_{2r}}\left(1+\oh{8}[b]_{0,\beta}r^{\beta-\gamma_{b}+p-s}\nr{u}_{L^{\infty}(\Omega_{0})}^{s-p}\right)\left |\frac{u-(u)_{2r}}{2r} \right |^{p} \ dx\\
&+c\mint_{B_{2r}}\left(a(x)-a_{i}(B_{2r})\right)\left |\frac{u-(u)_{2r}}{2r} \right|^{q}+a_{i}(B_{2r})\left |\frac{u-(u)_{2r}}{2r} \right|^{q} \ dx\\
\le &c\mint_{B_{2r}}\left |\frac{u-(u)_{2r}}{2r} \right |^{p}+2a_{i}(B_{2r})\left |\frac{u-(u)_{2r}}{2r} \right |^{q} \ dx\le c_{3}\mint_{B_{2r}}H_{0}^{q}\left(\frac{u-(u)_{2r}}{\oh{2r}}\right) \ dx,
\end{flalign*}
where $c_{3}=c_{3}(n,p,q,s,[b]_{0,\beta},\beta, \nr{u}_{L^{\infty}(\Omega_{0})})$. Moreover, if $n<p(1+\delta_{g})$ we obtain
\begin{flalign*}
\mint_{B_{r}}H(x,Du) \ dx \le& c\mint_{B_{2r}}H\left(x,\frac{u-(u)_{2r}}{r}\right) \ dx\\
\le &c \mint_{B_{2r}}\left(1+\oh{8}[b]_{0,\beta}r^{\beta-\gamma_{b}+(\lambda_{g}-1)(s-p)}[u]_{C^{0,\lambda_{g}}(\Omega_{0})}^{s-p}\right)\left |\frac{u-(u)_{2r}}{2r} \right |^{p} \ dx\\
&+c\mint_{B_{2r}}\left(a(x)-a_{i}(B_{2r})\right)\left |\frac{u-(u)_{2r}}{2r} \right|^{q}+a_{i}(B_{2r})\left |\frac{u-(u)_{2r}}{2r} \right|^{q} \ dx\\
\le &c\mint_{B_{2r}}\left |\frac{u-(u)_{2r}}{2r} \right |^{p}+\oh{2a_{i}(B_{2r})}\left |\frac{u-(u)_{2r}}{2r} \right |^{q} \ dx\le c_{3}\mint_{B_{2r}}H_{0}^{q}\left(\frac{u-(u)_{2r}}{\oh{2r}}\right) \ dx,
\end{flalign*}
with $c_{3}=c_{3}(n,p,q,s,[b]_{0,\beta},\beta, [u]_{C^{0,\lambda_{g}}(\Omega_{0})})$.\\
Finally, if $\texttt{ndeg}(B_{r})$ \oh{holds, then} by \eqref{cacc}, \eqref{mathAB}, the fact that either if $n\ge p(1+\delta_{g})$ or if $n<p(1+\delta_{g})$, $\alpha \ge \alpha-\gamma_{a}$ and $\beta\ge \beta-\gamma_{b}$, and the very definition of $\texttt{ndeg}(B_{r})$ we have
\begin{flalign*}
\mint_{B_{r}}H(x,Du) \ dx \le& c\mint_{B_{2r}}H\left(\frac{u-(u)_{2r}}{\oh{r}}\right) \ dx\\
\le &c\mint_{B_{2r}}\left | \frac{u-(u)_{2r}}{2r}\right |^{p}+\left(a(x)-a_{i}(B_{r})\right)\left |\frac{u-(u)_{2r}}{2r} \right |^{(q-p)+p}\\
&+ (b(x)-b_{i}(B_{r}))\left |\frac{u-(u)_{2r}}{2r}\right|^{(s-p)+p}+a_{i}(B_{r})\left |\frac{u-(u)_{2r}}{2r}\right|^{q}+b_{i}(B_{r})\left |\frac{u-(u)_{2r}}{2r}\right|^{s} \ dx \\
\le & c\mint_{B_{2r}}\left |\frac{u-(u)_{2r}}{2r} \right |^{p}+[a]_{0,\alpha}(\oh{4r})^{\alpha}\left |\frac{u-(u)_{2r}}{2r} \right |^{q}+[b]_{0,\beta}(\oh{4r})^{\beta}\left |\frac{u-(u)_{2r}}{2r} \right |^{s} \ dx\\
&+ c\mint_{B_{2r}}H_{0}\left(\frac{u-(u)_{2r}}{2r}\right) \ dx \le c_{4} \mint_{B_{2r}}H_{0}\left(\frac{u-(u)_{2r}}{2r}\right) \ dx,
\end{flalign*}
with $c_{4}=c_{4}(n,p,q,s,[a]_{0,\alpha},[b]_{0,\beta},\alpha,\beta)$.
\end{proof}

We conclude this section by recalling a quantitative Harmonic-approximation type result from \cite{bcm}. We shall report it in the form that better fits our necessities.
\begin{lemma}\label{quahar}
Let $B_{r}\subset \mathbb{R}^{n}$ be a ball, $\varepsilon \in (0,1)$, $\tilde{H}$ be one of the Young functions defined in \eqref{yfs} and $v \in W^{1,\tilde{H}}(B_{2r})$ be a map satisfying the following estimates:
\begin{flalign}
\mint_{B_{2r}}\tilde{H}(Dv) \ dx \le \tilde{c}_{1},\label{har1}
\end{flalign}
and
\begin{flalign}\label{har2}
\mint_{B_{r}}\tilde{H}(Dv)^{1+\sigma_{0}} \ dx \le \tilde{c}_{2},
\end{flalign}
where $\tilde{c}_{1},\tilde{c}_{2}\ge 1$ and $\sigma_{0}>0$ are fixed constants. Moreover, assume that
\begin{flalign}\label{har3}
\left | \ \mint_{B_{r}}D\tilde{H}(Dv)\cdot D\varphi \ dx\ \right |\le \varepsilon \nr{D\varphi}_{L^{\infty}(B_{r})} \ \ \mathrm{for \ all\ }\varphi \in C^{\infty}_{c}(B_{r}),
\end{flalign}
for some $\varepsilon \in (0,1)$. Then there exists a function $\tilde{h}\in W^{1,\tilde{H}}_{v}(B_{r})$ such that the following conditions are satisfied:
\begin{flalign}
&\mint_{B_{r}}D\tilde{H}(D\tilde{h})\cdot D\varphi \ dx =0 \ \ \mathrm{for \ all \ }\varphi \in C^{\infty}_{c}(B_{r}),\label{har4}\\
&\mint_{B_{r}}\tilde{H}(D\tilde{h})^{1+\oh{\sigma_{1}}} \ dx \le c(n,p,q,s,\sigma_{0})\tilde{c}_{2},\label{har5}\\
&\oh{\mint_{B_{r}}\widetilde{\mathcal{V}}(Dv,D\tilde{h})^{2} \ dx \le c\varepsilon ^{m}}\label{har6},
\end{flalign}
where \oh{$\tilde{\mathcal{V}}$ is the corresponding auxiliary function defined in \eqref{vfs}, $c=c(n,p,q,s,\tilde{c}_{1},\tilde{c}_{2})$, $\sigma_{1}=\sigma_{1}(n,p,q,s,\sigma_{0}) \in (0,\sigma_{0})$, $m=m(n,p,q,s,\sigma_{0})>0$.} 
\end{lemma}
\begin{proof}
The proof for $\tilde{H}=H_{0}^{p}, H_{0}^{q}, H_{0}^{s}$ is contained in \oh{\cite[Lemma 1]{bcm}}, so we focus on $\tilde{H}=H_{0}$. The proof we provide is in some sense a simplified version of the original one since we do not need a powerful result such as Theorem 5.1 from \cite{colmin2}. In fact we can recover some extra boundary integrability from Lemma \ref{L5}.\\
Define $h_{0}$ to be the solution to the Dirichlet problem
\begin{flalign*}
h_0 \mapsto \min_{w \in W^{1,H_{0}}_{v}(B_{r})}\int_{B_{r}}H_{0}(Dw) \ dx.
\end{flalign*}
By minimality \eqref{har4} is verified, since it is the Euler-Lagrange equation associated to \oh{the above variational problem}. Moreover, \oh{it follows from \eqref{har1} that}
\begin{flalign}\label{h3}
\mint_{B_{r}}H_{0}(Dh_{0}) \ dx \le \mint_{B_{r}}H_{0}(Dv) \ dx \le 2^{n}\tilde{c}_{1}.
\end{flalign}
Now, by the previous inequality, Lemma \ref{L5} with $a(\cdot)\equiv \const$ and $b(\cdot)\equiv \const$\oh{, and \eqref{har2},} we obtain
\begin{flalign}\label{gehc}
\mint_{B_{r}}H_{0}(Dh_{0})^{1+\sigma_{g}} \ dx\le & c\left\{\left(\mint_{B_{r}}H_{0}(Dh_{0}) \ dx\right)^{1+\sigma_{g}}+\mint_{B_{r}}H_{0}(Dv)^{1+\sigma_{g}} \ dx\right\}\le c\left((2^{n}\tilde{c}_{1})^{1+\sigma_{g}}+\tilde{c}_{2}\right)\oh{=:}\tilde{c}_{3}
\end{flalign}
\oh{for some $0<\sigma_{g}<\sigma_{0}$, which is \eqref{har5} with $\sigma_{1}=\sigma_{g}$. Here $\tilde{c}_{3}=\tilde{c}_{3}(n,p,q,s,\sigma_{g},\tilde{c}_{1},\tilde{c}_{2})$}.\\
Set $w=h_{0}-v\in W^{1,H_{0}}_{0}(B_{r})$ and let $\lambda \ge 1$ to be fixed later and consider $w_{\lambda}\in W^{1,\infty}_{0}(B_{r})$, the Lipschitz truncation of $w$ given by the main result in \cite{acefus1} and satisfying
\begin{flalign}\label{lips}
\nr{Dw_{\lambda}}_{L^{\infty}(B_{r})}\le c(n)\lambda\quad \mathrm{and}\quad \{w_{\lambda} \not = w\}\subset \{M(\snr{Dw})>\lambda\}\cup negligible \ set.
\end{flalign}
Using such properties, the fact that $t\mapsto H_{0}(t)$ is increasing, Markov's inequality, \eqref{har2}, \eqref{gehc} and the maximal theorem we deduce that
\oh{
\begin{flalign}\label{h1}
\frac{\snr{\{w_{\lambda}\not = w\}}}{\snr{B_{r}}}&\le \frac{\snr{B_{r}\cap\{M(\snr{Dw})\ge \lambda\}}}{\snr{B_{r}}}\nonumber \\
&\le\frac{1}{H_{0}(\lambda)^{1+\sigma_{g}}}\mint_{B_{r}}H_{0}(M(\snr{Dw}))^{1+\sigma_{g}} \ dx\nonumber \\
&\le\frac{c}{H_{0}(\lambda)^{1+\sigma_{g}}}\mint_{B_{r}}H_{0}(Dw)^{1+\sigma_{g}} \ dx\nonumber\\
&\le\frac{c}{H_{0}(\lambda)^{1+\sigma_{g}}}\mint_{B_{r}}H_{0}(Dh_{0})^{1+\sigma_{g}}+H_{0}(Dv)^{1+\sigma_{g}} \ dx \nonumber\\
&\le\frac{c}{H_{0}(\lambda)^{1+\sigma_{g}}} \left[\mint_{B_{r}}H_{0}(Dh_{0})^{1+\sigma_{g}} \ dx + \left(\mint_{B_{r}} H_{0}(Dv)^{1+\sigma_{0}} \ dx \right)^{\frac{1+\sigma_{g}}{1+\sigma_{0}}} \right]\nonumber\\
&\le \frac{c(\tilde{c}_{3}+\tilde{c}_{2}^{\frac{1+\sigma_{g}}{1+\sigma_{0}}})}{H_{0}(\lambda)^{1+\sigma_{g}}} \le \frac{c(\tilde{c}_{3}+\tilde{c}_{2})}{H_{0}(\lambda)^{1+\sigma_{g}}}, 
\end{flalign}}
where $c=c(n,p,q,s,\sigma_{g},\sigma_{0})$.\\
Now we test \eqref{har4} against $w_{\lambda}$, \oh{which is admissible by density,} to get
\oh{
\begin{flalign*}
(\mathrm{I})=&\mint_{B_{r}}(DH_{0}(Dh_{0})-DH_{0}(Dv))\cdot Dw_{\lambda}\chi_{\{w_{\lambda}=w\}} \ dx\\
&=-\mint_{B_{r}}DH_{0}(Dv)\cdot Dw_{\lambda} \ dx-\mint_{B_{r}}(DH_{0}(Dh_{0})-DH_{0}(Dv))\cdot Dw_{\lambda}\chi_{\{w\not =w_{\lambda}\}} \ dx =(\mathrm{II})+(\mathrm{III}).
\end{flalign*}}
The properties of $H_{0}$ and \oh{\eqref{control_2}} give
\begin{flalign}\label{h4}
(\mathrm{I})\ge c\mint_{B_{r}}\cd{\mathcal{V}_{0}(Dv,Dh_{0})}\chi_{\{w_{\lambda}=w\}} \ dx,
\end{flalign}
where \oh{$c=c(n,p,q,s)>0$}. Moreover, by \eqref{har3} and $\eqref{lips}_{1}$ we see that
\begin{flalign}\label{h0}
\oh{\snr{(\mathrm{II})}}\le c\varepsilon \lambda,
\end{flalign}
with $c=c(n)$. Before estimating term $(\mathrm{III})$, we recall a standard Young type inequality holding for $H_{0}$, see \cite{bcm}: for all $\sigma \in (0,1)$,
\begin{flalign}\label{yi}
xy\le \sigma^{1-s}H_{0}(x)+\sigma H^{*}_{0}(y),
\end{flalign}
where $H^{*}_{0}(y)=\sup_{x>0}\{yx-H_{0}(x)\}$ is the convex conjugate of $H_{0}$. Furthermore, there holds: $H_{0}^{*}\left(\frac{H_{0}(t)}{t}\right)\le H_{0}(t)$, see \cite{barlin} for more details. Now, using $\eqref{lips}_{1}$, \eqref{har1}, \eqref{h3}, \eqref{yi} and \eqref{h1} we estimate, for a certain fixed $\sigma \in (0,1)$,
\begin{flalign}\label{h2}
\snr{(\mathrm{III})}\le & s\nr{Dw_{\lambda}}_{L^{\infty}(B_{r})}\mint_{B_{r}}\left(\frac{H_{0}(Dh_{0})}{\snr{Dh_{0}}}+\frac{H_{0}(Dv)}{\snr{Dv}}\right)\chi_{\{w_{\lambda}\not =w\}} \ dx\nonumber \\
\le &\sigma \mint_{B_{r}}\left[H_{0}^{*}\left(\frac{H_{0}(Dh_{0})}{\snr{Dh_{0}}}\right)+H_{0}^{*}\left(\frac{H_{0}(Dv)}{\snr{Dv}}\right)\right] \ dx+\frac{cH_{0}(\nr{Dw_{\lambda}}_{L^{\infty}(B_{r})})}{\sigma^{s-1}}\frac{\snr{\{w_{\lambda}\not =w \}}}{\snr{B_{r}}}\nonumber \\
\le &\sigma \mint_{B_{r}}H_{0}(Dh_{0})+H_{0}(Dv) \ dx +\frac{c}{\sigma^{s-1}H_{0}(\lambda)^{\sigma_{g}}}\le \oh{2^{n+1}}\sigma\tilde{c}_{1}+\frac{c}{\sigma^{s-1}\lambda^{p\sigma_{g}}},
\end{flalign}
where we also used the fact that \oh{$H_{0}(\lambda)\ge \lambda^{p}$ since $\lambda \ge 1$.} Here $c=c(n,p,q,s,\sigma_{g})$.\\
Collecting \eqref{h4}, \eqref{h0} and \eqref{h2} we obtain
\begin{flalign}\label{h5}
\mint_{B_{r}}\oh{\mathcal{V}_{0}(Dv,Dh_{0})^{2}}\chi_{\{w_{\lambda}=w\}} \ dx\le &c\left(\varepsilon \lambda+\sigma+\sigma^{1-s}\lambda^{-p\sigma_{g}}\right),
\end{flalign}
for $c=c(\tilde{c}_{1},\tilde{c}_{2},\tilde{c}_{3},n,p,q,s,\sigma_{g})$ and $\sigma\in (0,1)$ to be fixed. For $\theta \in (0,1)$, by H\"older's inequality and \eqref{h5} we estimate
\begin{flalign*}
\left(\mint_{B_{r}}\mathcal{V}(Dv,Dh_{0})^{2\theta}\chi_{\{w_{\lambda} = w\}} \ dx\right)^{\frac{1}{\theta}}\le c\left(\varepsilon \lambda+\sigma+\sigma^{1-s}\lambda^{-p\sigma_{g}}\right).
\end{flalign*}
Again, by H\"older's inequality, \eqref{h1}, \eqref{har1} and \eqref{h3} we have
\begin{flalign}\label{h6}
\left(\mint_{B_{r}}\mathcal{V}(Dv,Dh_{0})^{2\theta}\chi_{\{w_{\lambda}\not = w\}} \ dx\right)^{\frac{1}{\theta}}\le &c\left(\frac{\snr{\{w_{\lambda}\not = w\}}}{\snr{B_{r}}}\right)^{\frac{1-\theta}{\theta}}\mint_{B_{r}}\mathcal{V}(Dv,D_{0})^{2}\ dx\nonumber \\
\le &cH_{0}(\lambda)^{-\frac{(1+\sigma_{g})(1-\theta)}{\theta}}\mint_{B_{r}}H_{0}(Dh_{0})+H_{0}(Dv) \ dx\le c\lambda^{-\frac{p(1-\theta)}{\theta}},
\end{flalign}
where $c=c(\tilde{c}_{1},\tilde{c}_{2},\tilde{c}_{3},n,p,q,s,\sigma_{g})$. Choosing in \eqref{h5} and \eqref{h6} $\lambda=\varepsilon^{-1/2}$ and $\sigma=\varepsilon^{\frac{3p\sigma_{g}}{4(s-1)}}$ we obtain that
\begin{flalign}\label{h7}
\left(\mint_{B_{r}}\mathcal{V}(Dv,Dh_{0})^{2\theta} \ dx\right)^{\frac{1}{\theta}}\le c\varepsilon^{\oh{2}m},
\end{flalign}
with $c=c(\tilde{c}_{1},\tilde{c}_{2},\tilde{c}_{3},n,p,q,s,\sigma_{g})$ and $m=\oh{\frac{1}{2}}\min\left\{\frac{1}{2},\frac{p\sigma_{g}}{4},\frac{\oh{3}p\sigma_{g}}{4(s-1)},\frac{p(1-\theta)}{2\theta}\right\}$. Notice that in the above estimates we still have a degree of freedom in $\theta$. Applying H\"older's inequality with exponents $\oh{\frac{2(1+\sigma_{g})}{1+2\sigma_{g}}}$ and $2(1+\sigma_{g})$ we obtain
\begin{flalign*}
\mint_{B_{r}}\mathcal{V}(Dv,Dh_{0})^{2} \ dx \le &\left(\mint_{B_{r}}\mathcal{V}(Dv,Dh_{0})^{\frac{2(1+\sigma_{g})}{1+2\sigma_{g}}} \ dx\right)^{\frac{1+2\sigma_{g}}{2(1+\sigma_{g})}}\left(\mint_{B_{r}}\mathcal{V}(Dv,Dh_{0})^{2(1+\sigma_{g})}\ dx\right)^{\frac{1}{1+\sigma_{g}}}\\
\le& c\varepsilon^{m}\left(\mint_{B_{r}}(H_{0}(Dh_{0})+H_{0}(Dv))^{1+\sigma_{g}} \ dx\right)^{\frac{1}{1+\sigma_{g}}}\le c\varepsilon^{m},
\end{flalign*}
with $c=c(\tilde{c}_{1},\tilde{c}_{2},\tilde{c}_{3},n,p,q,s,\sigma_{g})$. Here we used \eqref{har2}, \eqref{gehc}, and \eqref{h7} with $\theta=\frac{1+\sigma_{g}}{1+2\sigma_{g}}<1$. Recalling \eqref{control}, \oh{we can conclude from the previous estimate that
\begin{flalign*}
\mint_{B_{r}} \mathcal{V}(Dv,Dh_{0})^{2} \ dx \le c\varepsilon^{m},
\end{flalign*}
}
which is what we wanted. 
\end{proof}

\section{Morrey decay and Theorem \ref{T0}}
The proof of Theorem \ref{T0} goes in two moments: first, we prove that a suitable manipulation of a local minimizer $u$ of \eqref{hmvp} satisfies the assumptions of Lemma \ref{quahar}, then we exploit this to start an iteration which will eventually render the announced decay.\\\\

\emph{Step 1: Quantitative harmonic approximation}. Define the quantities
\begin{flalign*}
E=E(u,B_{2r})=\left(\mint_{B_{2r}}H(x,Du) \ dx \right)^{\frac{1}{p}}\quad \mathrm{and}\quad v=\frac{u}{E},
\end{flalign*}
where $u \in W^{1,H(\cdot)}_{\mathrm{loc}}(\oh{\Omega})$ is a local minimizer of \eqref{hmvp} and $B_{2r}\Subset \Omega_{0}\Subset \Omega$ is any ball of radius $r\le \frac{1}{2}$. From now on, \oh{we will consider the following auxiliary Young functions}
\begin{flalign*}
\begin{cases}
\ H_{0}(z)=\snr{z}^{p}+a_{i}(B_{2r})\snr{z}^{q}+b_{i}(B_{2r})\snr{z}^{s}, \quad &\tilde{H}_{0}(z)=\snr{z}^{p}+a_{i}(B_{2r})E^{q-p}\snr{z}^{q}+b_{i}(B_{2r})E^{s-p}\snr{z}^{s},\\
\ H_{0}^{s}(z)=\snr{z}^{p}+b_{i}(B_{2r})\snr{z}^{s},  \quad &\tilde{H}_{0}^{s}(z)=\snr{z}^{p}+b_{i}(B_{2r})E^{s-p}\snr{z}^{s},\\
\ H_{0}^{q}(z)=\snr{z}^{p}+a_{i}(B_{2r})\snr{z}^{q},  \quad &\tilde{H}_{0}^{q}(z)= \snr{z}^{p}+a_{i}(B_{2r})E^{q-p}\snr{z}^{q},\\
\ H_{0}^{p}(z)=\snr{z}^{p},
\end{cases}
\end{flalign*}
\oh{and
\begin{flalign*}
\begin{cases}
\mathcal{V}_{0}(z_{1},z_{2})^{2}=\snr{V_{p}(z_{1})-V_{p}(z_{2})}^{2}+a_{i}(B_{2r})\snr{V_{q}(z_{1})-V_{q}(z_{2})}^{2}+b_{i}(B_{2r})\snr{V_{s}(z_{1})-V_{s}(z_{2})}^{2},\\
\mathcal{V}_{0}^{s}(z_{1},z_{2})^{2}=\snr{V_{p}(z_{1})-V_{p}(z_{2})}^{2}+b_{i}(B_{2r})\snr{V_{s}(z_{1})-V_{s}(z_{2})}^{2},\\
\mathcal{V}_{0}^{q}(z_{1},z_{2})^{2}=\snr{V_{p}(z_{1})-V_{p}(z_{2})}^{2}+a_{i}(B_{2r})\snr{V_{q}(z_{1})-V_{q}(z_{2})}^{2},\\
\mathcal{V}_{0}^{p}(z_{1},z_{2})^{2}=\snr{V_{p}(z_{1})-V_{p}(z_{2})}^{2},
\end{cases}
\end{flalign*}
}
where $a_{i}(\cdot)$ and $b_{i}(\cdot)$ are defined as in \eqref{infab}. 
\oh{Since $u$ is a local minimizer of \eqref{hmvp}, a straightforward computation shows that $v$ is a local minimizer of the functional}
\begin{flalign*}
\tilde{\mathcal{H}}(w,\Omega)=\int_{\Omega}\snr{Dw}^{p}+a(x)E^{q-p}\snr{Dw}^{q}+b(x)E^{s-p}\snr{Dw}^{s} \ dx.
\end{flalign*}
\oh{Then, by scaling, it is easy to see that Lemma \ref{L3} holds true also for $v$ with the same extra integrability exponent $\delta_{g}=\delta_{g}(n,p,q,s,[a]_{0,\alpha},[b]_{0,\beta},\nr{Du}_{L^{p}(B_{r})})$ as $u$}. For any open $U\Subset \Omega$ it satisfies the Euler-Lagrange equation
\begin{flalign}\label{elv}
0=\int_{U}\left(p\snr{Dv}^{p-2}+qa(x)E^{q-p}\snr{Dv}^{q-2}+sb(x)E^{s-p}\snr{Dv}^{s-2}\right)Dv\cdot D\varphi \ dx \quad \mathrm{for \ all \ }\varphi \in C^{\infty}_{c}(U).
\end{flalign}
Moreover, if $\tilde{H}$ denotes $H_{0}^{p}$, $\tilde{H}_{0}^{q}$, $\tilde{H}_{0}^{s}$ or $\tilde{H}_{0}$, \oh{we see from the definition of $v$ that}
\begin{flalign}\label{h9}
\mint_{B_{2r}}\tilde{H}(Dv) \ dx \le E^{-p}\mint_{B_{2r}}H(x,Du) \ dx \le 1,
\end{flalign}
which is \eqref{har1} and, by \eqref{h9} and Lemma \ref{L3} we obtain, \oh{for some $\cd{\tilde{\sigma}_{g}}\in (0,\cd{\delta_{g}})$},
\begin{flalign}\label{h13}
\mint_{B_{r}}\tilde{H}(Dv)^{1+\cd{\tilde{\sigma}_{g}}} \ dx \le &\cd{\mint_{B_{r}}H(x,Dv)^{1+\tilde{\sigma}_{g}} \ dx=\left(\mint_{B_{2r}}H(x,Du) \ dx\right)^{-(1+\tilde{\sigma}_{g})}\left(\mint_{B_{r}}H(x,Du)^{1+\tilde{\sigma}_{g}} \ dx \right) }\le c,
\end{flalign}
where $c=c(n,p,q,s,[a]_{0,\alpha},[b]_{0,\beta},\cd{\nr{Du}_{L^{p}(\Omega_{0})}})$ is the constant appearing in Lemma \ref{L3} and this verifies \eqref{har2}. So we see that conditions \eqref{har1}-\eqref{har2} of Lemma \ref{quahar} are \oh{matched} with $\sigma_{0}= \cd{\tilde{\sigma}_{g}}$ no matter what degeneracy (or non degeneracy) condition holds on $B_{2r}$. 
Here \cd{$\sigma_{g}$} is the exponent given by Lemma \ref{L5} depending on \oh{whether} $\tilde{H}$ denotes $\tilde{H}_{0}^{p}$, $\tilde{H}_{0}^{q}$, $\tilde{H}_{0}^{s}$ or $\tilde{H}_{0}$. Clearly we have no problem of integrability, since \cd{$\tilde{\sigma}_{g}<\delta_{g}$}, which is the corresponding exponent coming from Lemma \ref{L3}. \oh{We now define
\begin{flalign}\label{diff}
\sigma_{a}=\alpha-\gamma_{a}-\frac{n(q-p)}{p(1+\delta_{g})}\quad \mathrm{and}\quad\sigma_{b}=\beta-\gamma_{b}-\frac{n(s-p)}{p(1+\delta_{g})}.
\end{flalign}
A simple computation shows that $\sigma_{a}$ and $\sigma_{b}$ are both positive numbers}.\\
\oh{We first assume} $\texttt{deg}(B_{2r})$. From \eqref{elv} we deduce that
\begin{flalign*}
\left |\ \mint_{B_{r}}DH_{0}^{p}(Dv)\cdot D\varphi \ dx \  \right |\le & qE^{q-p}\mint_{B_{r}}a(x)\snr{Dv}^{q-1}\snr{D\varphi} \ dx+sE^{s-p}\mint_{B_{r}}b(x)\snr{Dv}^{s-1}\snr{D\varphi} \ dx \oh{=:(\mathrm{I})_{\texttt{deg}}+(\mathrm{II})_{\texttt{deg}}.}
\end{flalign*}
From the very definition of condition $\texttt{deg}$, Lemma \ref{L3}, \eqref{h9}, H\"older's inequality \oh{and \eqref{diff}} we get
\begin{flalign}\label{h8}
\oh{(\mathrm{I})_{\texttt{deg}}}\le& \cd{cE^{\frac{q-p}{q}}r^{\frac{\alpha-\gamma_{a}}{q}}\nr{D\varphi}_{L^{\infty}(B_{r})}\mint_{B_{r}}(E^{q-p}a(x))^{\frac{q-1}{q}}\snr{Dv}^{q-1} \ dx}\nonumber \\
\le &\cd{c\nr{D\varphi}_{L^{\infty}(B_{r})}\left(\mint_{B_{2r}}H(x,Du) \ dx\right)^{\frac{q-p}{pq}}r^{\frac{\alpha-\gamma_{a}}{q}}\left(\mint_{B_{r}}E^{q-p-q}a(x)\snr{Du}^{q} \ dx\right)^{\frac{q-1}{q}}}\nonumber \\
\le &\cd{c\nr{D\varphi}_{L^{\infty}(B_{r})}\nr{H(\cdot,Du)}_{L^{1+\delta_{g}}(\Omega_{0})}^{\frac{q-p}{pq}}r^{\frac{\alpha-\gamma_{a}}{q}-\frac{n(q-p)}{pq(1+\delta_{g})}}\le c\nr{D\varphi}_{L^{\infty}(B_{r})}r^{\frac{\sigma_{a}}{q}}}
\end{flalign}
with $c_{1}=c_{1}(n,p,q,[a]_{0,\alpha},\alpha,\nr{H(\cdot,Du)}_{L^{1+\delta_{g}}(\Omega_{0})})$. In a totally similar way we obtain
\begin{flalign}\label{h10}
\oh{(\mathrm{II})_{\texttt{deg}}}\le& \cd{cE^{\frac{s-p}{s}}r^{\frac{\beta-\gamma_{b}}{s}}\nr{D\varphi}_{L^{\infty}(B_{r})}\mint_{B_{r}}(E^{s-p}b(x))^{\frac{s-1}{s}}\snr{Dv}^{s-1} \ dx}\nonumber \\
\le &\cd{c\nr{D\varphi}_{L^{\infty}(B_{r})}\left(\mint_{B_{2r}}H(x,Du) \ dx\right)^{\frac{s-p}{ps}}r^{\frac{\beta-\gamma_{b}}{s}}\left(\mint_{B_{r}}E^{s-p-s}b(x)\snr{Du}^{s} \ dx\right)^{\frac{s-1}{s}}}\nonumber \\
\le &\cd{c\nr{D\varphi}_{L^{\infty}(B_{r})}\nr{H(\cdot,Du)}_{L^{1+\delta_{g}}(\Omega_{0})}^{\frac{s-p}{ps}}r^{\frac{\beta-\gamma_{b}}{s}-\frac{n(s-p)}{ps(1+\delta_{g})}}\le c\nr{D\varphi}_{L^{\infty}(B_{r})}r^{\frac{\sigma_{b}}{s}}}
\end{flalign}
where $c_{2}=c_{2}(n,p,s,[b]_{0,\beta},\beta,\nr{H(\cdot,Du)}_{L^{1+\delta_{g}}(\Omega_{0})})$.\\ 
Now we \cd{define $\tilde{\sigma}_{p}:=\frac{1}{2}\min\{q^{-1}\sigma_{a},s^{-1}\sigma_{b}\}>0$ and} fix a threshold radius $\tilde{R}^{1}_{*}$ such that $\max\{c_{1},c_{2}\}(\tilde{R}^{1}_{*})^{\tilde{\sigma_{\oh{p}}}}\le\frac{1}{2}$ and assume that $0<r\le \min\{\tilde{R}^{1}_{*},1\}$. In correspondence of such a choice, by \eqref{h8} and \eqref{h10} we can conclude that
\begin{flalign}\label{h11}
\left |\ \mint_{B_{r}}DH_{0}^{p}(Dv)\cdot D\varphi \ dx \  \right |\le &r^{\tilde{\sigma}_{p}}\nr{D\varphi}_{L^{\infty}(B_{r})},
\end{flalign}
so the assumptions of Lemma \ref{quahar} are matched and there exists a $H_{0}^{p}$-harmonic map $\tilde{h}_{p}$ satisfying in particular \eqref{har6}. It is clear that, if $h_{p}=E\tilde{h}_{p}$, then $h_{p}$ is still $H_{0}^{p}$-harmonic, $\left.h_{p} \right |_{\partial B_{r}}=\left.u \right |_{\partial B_{r}}$ and, by \eqref{har6}, 
\begin{flalign}\label{h12}
\oh{\mint_{B_{r}}\mathcal{V}_{0}^{p}(Du,Dh_{p})^{2} \ dx\le cr^{m_{p}}\mint_{B_{2r}}H(x,Du) \ dx},
\end{flalign}
where $c=c(n,p,q,s,[a]_{0,\alpha},[b]_{0,\beta})$ and $m_{p}=m_{p}(n,p,q,s,[a]_{0,\alpha},[b]_{0,\beta})$. 
Suppose now that $\texttt{deg}_{\alpha}(B_{2r})$ holds. Then, by \eqref{elv} we obtain
\begin{flalign*}
\left | \ \mint_{B_{r}}D\tilde{H}^{s}_{0}(Dv)\cdot D\varphi \ dx\ \right |\le& qE^{q-p}\mint_{B_{r}}a(x)\snr{Dv}^{q-1}\snr{D\varphi} \ dx +sE^{s-p}\mint_{B_{r}}\left(b(x)-\oh{b_{i}(B_{2r})}\right)\snr{Dv}^{s-1}\snr{D\varphi} \ dx\\
\le& \oh{(\mathrm{I})_{\texttt{deg}_{\alpha}}+(\mathrm{II})_{\texttt{deg}_{\alpha}}.}
\end{flalign*}
As before we estimate
\begin{flalign}\label{h14}
\oh{(\mathrm{I})_{\texttt{deg}_{\alpha}}}\le &\cd{c\nr{D\varphi}_{L^{\infty}(B_{r})}\left(\mint_{B_{2r}}H(x,Du) \ dx\right)^{\frac{q-p}{pq}}r^{\frac{\alpha-\gamma_{a}}{q}}\left(\mint_{B_{r}}E^{q-p-q}a(x)\snr{Du}^{q} \ dx\right)^{\frac{q-1}{q}}}\nonumber \\
\le &\cd{c\nr{D\varphi}_{L^{\infty}(B_{r})}\nr{H(\cdot,Du)}_{L^{1+\delta_{g}}(\Omega_{0})}^{\frac{q-p}{pq}}r^{\frac{\alpha-\gamma_{a}}{q}-\frac{n(q-p)}{pq(1+\delta_{g})}}\le c\nr{D\varphi}_{L^{\infty}(B_{r})}r^{\frac{\sigma_{a}}{q}}}
\end{flalign}
with $c_{1}=c_{1}(n,p,q,[a]_{0,\alpha},\alpha,\nr{H(\cdot,Du)}_{L^{1+\delta_{g}}(\Omega_{0})})$, and
\begin{flalign}\label{h15}
\oh{(\mathrm{II})_{\texttt{deg}_{\alpha}}}\le& \cd{c\nr{D\varphi}_{L^{\infty}(B_{r})}E^{\frac{s-p}{s}}r^{\frac{\beta}{s}+\frac{\gamma_{b}(s-1)}{s}}\left(\mint_{B_{r}}(E^{s-p}r^{\beta-\gamma_{b}})^{\frac{s-1}{s}}\snr{Dv}^{s-1} \ dx\right)}\nonumber \\
\le &\cd{c\nr{D\varphi}_{L^{\infty}(B_{r})}r^{\frac{\gamma_{b}(s-1)}{s}}r^{\frac{1}{s}\left(\beta-\frac{n(s-p)}{p(1+\delta_{g})}\right)}\left(\mint_{B_{2r}}E^{-p}b_{i}(B_{2r})\snr{Du}^{s} \ dx\right)^{\frac{s-1}{s}}\le cr^{\frac{\gamma_{b}(s-1)}{s}+\frac{1}{s}\left(\beta-\frac{n(s-p)}{p(1+\delta_{g})}\right)}\nr{D\varphi}_{L^{\infty}(B_{r})}},
\end{flalign}
where $c_{2}=c_{2}(n,p,s,[b]_{0,\beta},\beta,\nr{H(\cdot,Du)}_{L^{1+\delta_{g}}(\Omega_{0})})$.\\
\cd{Define $\tilde{\sigma}_{s}:=\frac{1}{2}\min\left\{\frac{\sigma_{a}}{q},\frac{\gamma_{b}(s-1)}{s}+\frac{1}{s}\left(\beta-\frac{n(s-p)}{p(1+\delta_{g})}\right)\right\}>0$ and} fix a threshold radius $\tilde{R}^{2}_{*}$ such that $\max\{c_{1},c_{2}\}(\tilde{R}^{2}_{*})^{\tilde{\sigma_{\oh{s}}}}\le\frac{1}{2}$ and assume that $0<r\le \min\{\tilde{R}^{1}_{*},\tilde{R}^{2}_{*},1\}$. In correspondence of such a choice, by \eqref{h14} and \eqref{h15} we can conclude that
\begin{flalign}\label{h16}
\left |\ \mint_{B_{r}}D\tilde{H}_{0}^{s}(Dv)\cdot D\varphi \ dx \  \right |\le &r^{\tilde{\sigma}_{s}}\nr{D\varphi}_{L^{\infty}(B_{r})},
\end{flalign}
so the assumptions of Lemma \ref{quahar} are matched and there exists a $\tilde{H}_{0}^{s}$-harmonic map $\tilde{h}_{s}$ satisfying in particular \eqref{har6}. Clearly, if $h_{s}=E\tilde{h}_{s}$, then $h_{s}$ is $H_{0}^{s}$-harmonic, $\left.h_{s} \right |_{\partial B_{r}}=\left.u \right |_{\partial B_{r}}$ and, by \eqref{har6}, 
\begin{flalign}\label{h17}
\oh{\mint_{B_{r}}\mathcal{V}^{s}_{0}(Du,Dh_{s})^{2} \ dx\le cr^{m_{s}}\mint_{B_{2r}}H(x,Du) \ dx},
\end{flalign}
where $c=c(n,p,q,s,[a]_{0,\alpha},[b]_{0,\beta})$ and $m_{s}=m_{s}(n,p,q,s,[a]_{0,\alpha},[b]_{0,\beta})$.\\
This time assume $\texttt{deg}_{\beta}(B_{2r})$ holds. Then, by \eqref{elv} we obtain
\begin{flalign*}
\left | \ \mint_{B_{r}}D\tilde{H}^{q}_{0}(Dv)\cdot D\varphi \ dx\ \right |\le& sE^{s-p}\mint_{B_{r}}b(x)\snr{Dv}^{s-1}\snr{D\varphi} \ dx +qE^{q-p}\mint_{B_{r}}\left(a(x)-\oh{a_{i}(B_{2r})}\right)\snr{Dv}^{q-1}\snr{D\varphi} \ dx\\
\le& \oh{(\mathrm{I})_{\texttt{deg}_{\beta}}+(\mathrm{II})_{\texttt{deg}_{\beta}}.}
\end{flalign*}
As above we estimate
\begin{flalign}\label{h18}
\oh{(\mathrm{I})_{\texttt{deg}_{\beta}}}\le& \cd{c\nr{D\varphi}_{L^{\infty}(B_{r})}E^{\frac{s-p}{s}}r^{\frac{\beta-\gamma_{b}}{s}}\left(\mint_{B_{r}}E^{-p}b(x)\snr{Du}^{s}\right)^{\frac{s-1}{s}}}\nonumber \\
\le &\cd{c\nr{D\varphi}_{L^{\infty}(B_{r})}\nr{H(\cdot,Du)}^{\frac{s-p}{sp}}_{L^{1+\delta_{g}}(\Omega_{0})}r^{\frac{1}{s}\left(\beta-\gamma_{b}-\frac{n(s-p)}{p(1+\delta_{g})}\right)}\le cr^{\frac{\sigma_{b}}{s}}\nr{D\varphi}_{L^{\infty}(B_{r})}},
\end{flalign}
with $c_{1}=c_{1}(n,p,s,[b]_{0,\beta},\beta,\nr{H(\cdot,Du)}_{L^{1+\delta_{g}}(\Omega_{0})})$, and
\begin{flalign}\label{h19}
\oh{(\mathrm{II})_{\texttt{deg}_{\beta}}}\le& \cd{c\nr{D\varphi}_{L^{\infty}(B_{r})}E^{\frac{q-p}{q}}r^{\frac{\alpha}{q}+\frac{\gamma_{a}(q-1)}{q}}\left(\mint_{B_{r}}(E^{q-p}r^{\alpha-\gamma_{a}})^{\frac{q-1}{q}}\snr{Dv}^{q-1} \ dx\right)}\nonumber \\
\le &\cd{c\nr{D\varphi}_{L^{\infty}(B_{r})}r^{\frac{\gamma_{a}(q-1)}{q}}r^{\frac{1}{q}\left(\alpha-\frac{n(q-p)}{p(1+\delta_{g})}\right)}\left(\mint_{B_{2r}}E^{-p}a_{i}(B_{2r})\snr{Du}^{q} \ dx\right)^{\frac{q-1}{q}}\le cr^{\frac{\gamma_{a}(q-1)}{q}+\frac{1}{q}\left(\alpha-\frac{n(q-p)}{p(1+\delta_{g})}\right)}\nr{D\varphi}_{L^{\infty}(B_{r})}},
\end{flalign}
where $c_{2}=c_{2}(n,p,q,[a]_{0,\alpha},\alpha,\nr{H(\cdot,Du)}_{L^{1+\delta_{g}}(\Omega_{0})})$.
\cd{Let $\tilde{\sigma}_{q}:=\frac{1}{2}\min\left\{\frac{\sigma_{b}}{s},\frac{\gamma_{a}(q-1)}{q}+\frac{1}{q}\left(\alpha-\frac{n(q-p)}{p(1+\delta_{g})}\right)\right\}>0$ and} fix a threshold radius $\tilde{R}^{3}_{*}$ such that $\max\{c_{1},c_{2}\}(\tilde{R}^{3}_{*})^{\tilde{\sigma}_{q}}\le\frac{1}{2}$ and assume that $0<r\le \min\{\tilde{R}^{1}_{*},\tilde{R}^{2}_{*},\tilde{R}^{3}_{*},1\}$. In correspondence of such a choice, by \eqref{h14} and \eqref{h15} we can conclude that
\begin{flalign}\label{h20}
\left |\ \mint_{B_{r}}D\tilde{H}_{0}^{q}(Dv)\cdot D\varphi \ dx \  \right |\le &r^{\tilde{\sigma}_{q}}\nr{D\varphi}_{L^{\infty}(B_{r})},
\end{flalign}
so the assumptions of Lemma \ref{quahar} are satisfied and there exists a $\tilde{H}_{0}^{q}$-harmonic map $\tilde{h}_{q}$ satisfying in particular \eqref{har6}. Clearly, if $h_{q}=E\tilde{h}_{q}$, then $h_{q}$ is $H_{0}^{q}$-harmonic, $\left.h_{q} \right |_{\partial B_{r}}=\left.u \right |_{\partial B_{r}}$ and, by \eqref{har6}, 
\begin{flalign}\label{h21}
\oh{\mint_{B_{r}}\oh{\mathcal{V}^{q}_{0}(Du,Dh_{q})}^{2} \ dx\le cr^{m_{q}}\mint_{B_{2r}}H(x,Du) \ dx},
\end{flalign}
where $c=c(n,p,q,s,[a]_{0,\alpha},[b]_{0,\beta})$ and $m_{q}=m_{q}(n,p,q,s,[a]_{0,\alpha},[b]_{0,\beta})$.\\
Finally, suppose $\texttt{ndeg}(B_{2r})$ holds. Then, by \eqref{elv} we obtain
\begin{flalign*}
\left | \ \mint_{B_{r}}D\tilde{H}_{0}(Dv)\cdot D\varphi \ dx\ \right |\le& sE^{s-p}\mint_{B_{r}}\left(b(x)-b_{i}(B_{2r})\right)\snr{Dv}^{s-1}\snr{D\varphi} \ dx +qE^{q-p}\mint_{B_{r}}\left(a(x)-\oh{a_{i}(B_{2r})}\right)\snr{Dv}^{q-1}\snr{D\varphi} \ dx\\
\le& \oh{(\mathrm{I})_{\texttt{ndeg}}+(\mathrm{II})_{\texttt{ndeg}}.}
\end{flalign*}
As above we estimate
\begin{flalign}\label{h22}
\oh{(\mathrm{I})_{\texttt{ndeg}}}\le& \cd{c\nr{D\varphi}_{L^{\infty}(B_{r})}E^{\frac{s-p}{s}}r^{\frac{\beta}{s}+\frac{\gamma_{b}(s-1)}{s}}\left(\mint_{B_{r}}(E^{s-p}r^{\beta-\gamma_{b}})^{\frac{s-1}{s}}\snr{Dv}^{s-1} \ dx\right)}\nonumber \\
\le &\cd{c\nr{D\varphi}_{L^{\infty}(B_{r})}r^{\frac{\gamma_{b}(s-1)}{s}}r^{\frac{1}{s}\left(\beta-\frac{n(s-p)}{p(1+\delta_{g})}\right)}\left(\mint_{B_{2r}}E^{-p}b_{i}(B_{2r})\snr{Du}^{s} \ dx\right)^{\frac{s-1}{s}}\le cr^{\frac{\gamma_{b}(s-1)}{s}+\frac{1}{s}\left(\beta-\frac{n(s-p)}{p(1+\delta_{g})}\right)}\nr{D\varphi}_{L^{\infty}(B_{r})}},
\end{flalign}
with $c_{1}=c_{1}(n,p,s,[b]_{0,\beta},\beta,\nr{H(\cdot,Du)}_{L^{1+\delta_{g}}(\Omega_{0})})$, and
\begin{flalign}\label{h23}
\oh{(\mathrm{II})_{\texttt{ndeg}}}\le& \cd{c\nr{D\varphi}_{L^{\infty}(B_{r})}E^{\frac{q-p}{q}}r^{\frac{\alpha}{q}+\frac{\gamma_{a}(q-1)}{q}}\left(\mint_{B_{r}}(E^{q-p}r^{\alpha-\gamma_{a}})^{\frac{q-1}{q}}\snr{Dv}^{q-1} \ dx\right)}\nonumber \\
\le &\cd{c\nr{D\varphi}_{L^{\infty}(B_{r})}r^{\frac{\gamma_{a}(q-1)}{q}}r^{\frac{1}{q}\left(\alpha-\frac{n(q-p)}{p(1+\delta_{g})}\right)}\left(\mint_{B_{2r}}E^{-p}a_{i}(B_{2r})\snr{Du}^{q} \ dx\right)^{\frac{q-1}{q}}\le cr^{\frac{\gamma_{a}(q-1)}{q}+\frac{1}{q}\left(\alpha-\frac{n(q-p)}{p(1+\delta_{g})}\right)}\nr{D\varphi}_{L^{\infty}(B_{r})}},
\end{flalign}
where $c_{2}=c_{2}(n,p,q,[a]_{0,\alpha},\alpha,\nr{H(\cdot,Du)}_{L^{1+\delta_{g}}(\Omega_{0})})$.\\
\cd{Let $\tilde{\sigma}_{0}:=\frac{1}{2}\min\left\{\frac{\gamma_{a}(q-1)}{q}+\frac{1}{q}\left(\alpha-\frac{n(q-p)}{p(1+\delta_{g})}\right),\frac{\gamma_{b}(s-1)}{s}+\frac{1}{s}\left(\beta-\frac{n(s-p)}{p(1+\delta_{g})}\right)\right\}>0$ and }fix another threshold radius $\tilde{R}^{4}_{*}$ such that $\max\{c_{1},c_{2}\}(\tilde{R}^{4}_{*})^{\tilde{\sigma}_{0}}\le\frac{1}{2}$ and assume that $0<r\le \min\{\tilde{R}^{1}_{*},\tilde{R}^{2}_{*},\tilde{R}^{3}_{*},\tilde{R}^{4}_{*},1\}$. In correspondence of such a choice, by \eqref{h22} and \eqref{h23} we can conclude that
\begin{flalign}\label{h24}
\left |\ \mint_{B_{r}}D\tilde{H}_{0}(Dv)\cdot D\varphi \ dx \  \right |\le &r^{\tilde{\sigma}_{0}}\nr{D\varphi}_{L^{\infty}(B_{r})},
\end{flalign}
so the assumptions of Lemma \ref{quahar} are satisfied and there exists a $\tilde{H}_{0}$-harmonic map $\tilde{h}_{0}$ satisfying in particular \eqref{har6}. Clearly, if $h_{0}=E\tilde{h}_{0}$, then $h_{0}$ is $H_{0}$-harmonic, $\left.h_{0} \right |_{\partial B_{r}}=\left.u \right |_{\partial B_{r}}$ and, by \eqref{har6}, 
\begin{flalign}\label{h25}
\mint_{B_{r}}&\oh{\mathcal{V}_{0}(Du,Dh_{0})^{2} \ dx}\le cr^{m_{0}}\mint_{B_{2r}}H(x,Du) \ dx,
\end{flalign}
where $c=c(n,p,q,s,[a]_{0,\alpha},[b]_{0,\beta})$ and $m_{0}=m_{0}(n,p,q,s,[a]_{0,\alpha},[b]_{0,\beta})$. Summarizing we got

\begin{flalign*}
&\cd{\texttt{deg}(B_{2r})\Rightarrow \mint_{B_{r}}\mathcal{V}_{0}^{p}(Du,Dh_{p})^{2} \ dx \le cr^{m_{p}}\mint_{B_{2r}}H(x,Du) \ dx} \\
&\cd{\texttt{deg}_{\alpha}(B_{2r})\Rightarrow \mint_{B_{r}}\mathcal{V}_{0}^{s}(Du,Dh_{s})^{2} \ dx \le cr^{m_{s}}\mint_{B_{2r}}H(x,Du) \ dx}\\
&\cd{\texttt{deg}_{\beta}(B_{2r})\Rightarrow \mint_{B_{r}}\mathcal{V}_{0}^{q}(Du,Dh_{q})^{2} \ dx \le cr^{m_{q}}\mint_{B_{2r}}H(x,Du) \ dx}\\
&\cd{\texttt{ndeg}(B_{2r})\Rightarrow \mint_{B_{r}}\mathcal{V}_{0}(Du,Dh_{0})^{2} \ dx \le cr^{m_{0}}\mint_{B_{2r}}H(x,Du) \ dx}
\end{flalign*}
where the above holds for $0<r\le \tilde{R}_{*}=\min\{\tilde{R}^{1}_{*},\tilde{R}^{2}_{*},\tilde{R}^{3}_{*},\tilde{R}^{4}_{*},1\}$, and all the quantities involved are as described before. Finally, for the sake of clarity, we let $m=\min\{m_{p},m_{q},m_{s},m_{0}\}$.
Now take a ball $B_{r}$ with $0<r\le \frac{1}{2}\tilde{R}_{*}$ such that $B_{2r}\Subset \Omega_{0}\Subset \Omega$. Fix $\tau_{p}\in \left(0,\frac{1}{8}\right)$ and assume $\texttt{deg}(B_{2r})$ and $\texttt{deg}(B_{4\tau_{p} r})$. We fix $\vartheta \in (0,n)$ and we estimate, by \eqref{degc}, Poincar\'e's inequality, Proposition \ref{p0} with $\varphi=H_{0}^{p}$, \eqref{control} and \eqref{h12},
\begin{flalign}\label{h26}
\int_{B_{2\tau_{p} r}}H(x,Du) \ dx \le & c\int_{B_{4\tau_{p}r}}H_{0}^{p}\left(\frac{u-(u)_{4\tau_{p}r}}{4\tau_{p}r}\right) \ dx\le \oh{c\int_{B_{4\tau_{p}r}}H_{0}^{p}(Du) \ dx}\nonumber \\
\le& c\left\{\oh{\int_{B_{4\tau_{p}r}}\mathcal{V}_{0}^{p}(Du,Dh_{p})^{2} \ dx}+\int_{B_{4\tau_{p}r}}\snr{V_{p}(Dh_{p})}^{2} \ dx\right\}\nonumber \\
\le& c\left\{\oh{\int_{B_{4\tau_{p}r}}\mathcal{V}_{0}^{p}(Du,Dh_{p})^{2} \ dx}+\oh{\snr{B_{4\tau_{p}r}}\sup_{B_{4\tau_{p}r}}H_{0}^{p}(Dh_{p})}\right\}\nonumber \\
\le &c\left\{\int_{B_{r}}\oh{\mathcal{V}_{0}^{p}(Du,Dh_{p})^{2}} \ dx+\tau^{n}_{p}\int_{B_{r}}H^{p}_{0}(Dh_{p}) \ dx\right\}\nonumber \\
\le & \tau^{n-\vartheta}_{p}\left(cr^{m}\tau^{\vartheta-n}+c\tau^{\vartheta}_{p}\right)\int_{B_{2r}}H(x,Du) \ dx,
\end{flalign}
where $c=c(\texttt{data}(\Omega_{0}),\vartheta)$. For the ease of exposition we set $2r=\rho$ and adjusting the constants in \eqref{h26} we get
\begin{flalign*}
\int_{B_{\tau_{p}\rho}}H(x,Du) \ dx \le \tau_{p}^{n-\vartheta}\left(c\rho^{m}\tau_{p}^{\vartheta-n}+c\tau^{\vartheta}_{p}\right)\int_{B_{\rho}}H(x,Du) \ dx.
\end{flalign*}
Selecting $\tau_{p}$ in such a way that $c\tau_{p}^{\vartheta}\le \frac{1}{2}$ and a threshold radius $R_{*}^{1}\in(0,\tilde{R}_{*}]$ such that $cR^{m}\tau_{p}^{\vartheta-n}\le\frac{1}{2}$, we can conclude that, for all $\rho\in (0,R^{1}_{*})$ and all $\vartheta \in (0,n)$,
\begin{flalign}\label{h27}
\int_{B_{\tau_{p}\rho}}H(x,Du) \ dx \le \tau_{p}^{n-\vartheta}\int_{B_{\rho}}H(x,Du) \ dx.
\end{flalign}
Now fix $\tau_{s}\in \left(0,\frac{1}{8}\right)$, assume $\texttt{deg}_{\alpha}(B_{2r})$ and that $a_{i}(B_{4\tau_{s}r})\le 4[a]_{0,\alpha}(4\tau_{s}r)^{\alpha-\gamma_{a}}$, where $r<\frac{1}{2}R^{1}_{*}$. For $\vartheta \in (0,n)$, by \eqref{degac}, Poincar\'e's inequality, Proposition \ref{p0} with $\varphi=H_{0}^{s}$, \eqref{control} and \eqref{h17} we obtain
\begin{flalign}\label{h28}
\int_{B_{2\tau_{s}r}}H(x,Du) \ dx \le & c\int_{B_{4\tau_{s}r}}H_{0}^{s}\left(\frac{u-(u)_{4\tau_{s}r}}{4\tau_{s}r}\right) \ dx\le \oh{\int_{B_{4\tau_{s}r}}H_{0}^{s}(Du) \ dx}\nonumber \\
\le& c\left\{\int_{B_{4}\tau_{s}r}\oh{\mathcal{V}_{0}^{s}(Du,Dh_{s})^{2}} \ dx+\int_{B_{4\tau_{s}r}}H_{0}^{s}(Dh_{s}) \ dx\right\}\nonumber \\
\le& c\left\{\int_{B_{4}\tau_{s}r}\oh{\mathcal{V}_{0}^{s}(Du,Dh_{s})^{2}} \ dx+\oh{\snr{B_{4\tau_{s}r}}\sup_{B_{4\tau_{s}r}}H_{0}^{s}(Dh_{s})}\right\}\nonumber \\
\le& c\left\{\int_{B_{r}}\oh{\mathcal{V}_{0}^{s}(Du,Dh_{s})^{2}} \ dx+\tau_{s}^{n}\int_{B_{r}}H_{0}^{s}(Dh_{s}) \ dx\right\}\nonumber \\
\le& \tau^{n-\vartheta}_{s}\left(cr^{m}\tau_{s}^{\vartheta-n}+c\tau_{s}^{\vartheta}\right)\int_{B_{2r}}H(x,Du) \ dx,
\end{flalign}
where $c=c(\texttt{data}(\Omega_{0}),\vartheta)$. Again, we name $\rho=2r$ thus getting
\begin{flalign*}
\int_{B_{\tau_{s}\rho}}H(x,Du) \ dx \le \tau_{s}^{n-\vartheta}\left(c\rho^{m}\tau_{s}^{\vartheta-n}+\tau^{\vartheta}_{s}\right)\int_{B_{\rho}}H(x,Du) \ dx,
\end{flalign*}
where, as before, $\vartheta \in (0,n)$ is arbitrary. Choose $\tau_{s}$ small enough so that $c\tau^{\vartheta}_{s}<\frac{1}{2}$ and a threshold $R^{2}_{*}$, $0<R^{2}_{*}\le R^{1}_{*}$ such that $c(R^{1}_{*})^{m}\tau^{\vartheta-n}_{s}\le \frac{1}{2}$. Hence, for all $\rho \in (0,R^{2}_{*}]$ and all $\vartheta \in (0,n)$ we get
\begin{flalign}\label{h29}
\int_{B_{\tau_{s}\rho}}H(x,Du) \ dx \le \tau_{s}^{n-\vartheta}\int_{B_{\rho}}H(x,Du) \ dx.
\end{flalign}
Consider $\tau_{q}\in \left(0,\frac{1}{8}\right)$, assume $\texttt{deg}_{\beta}(B_{2r})$ and that $b_{i}(B_{4\tau_{q}r})\le 4[b]_{0,\beta}(4\tau_{q}r)^{\beta-\gamma_{b}}$, where $r<\frac{1}{2}R^{2}_{*}$. For $\vartheta \in (0,n)$, by \eqref{degbc}, Poincar\'e's inequality, Proposition \ref{p0} with $\varphi=H_{0}^{q}$ and \eqref{h21} we obtain
\begin{flalign}\label{h30}
\int_{B_{2\tau_{q}r}}H(x,Du) \ dx \le & c\int_{B_{4\tau_{q}r}}H_{0}^{q}\left(\frac{u-(u)_{4\tau_{q}r}}{4\tau_{q}r}\right) \ dx\le \oh{c\int_{B_{4\tau_{s}t}}H_{0}^{q}(Du) \ dx}\nonumber \\
\le & c\left\{\int_{B_{4\tau_{q}r}}\oh{\mathcal{V}_{0}^{q}(Du,Dh_{q})^{2}} \ dx +\int_{B_{4\tau_{q}r}}H_{0}^{q}(Dh_{q}) \ dx\right\}\nonumber \\
\le & c\left\{\int_{B_{4\tau_{q}r}}\oh{\mathcal{V}_{0}^{q}(Du,Dh_{q})^{2}} \ dx +\oh{\snr{B_{4\tau_{q}r}}\sup_{B_{4\tau_{q}r}}H_{0}^{q}(Dh_{q})}\right\}\nonumber \\
\le & c\left\{\int_{B_{r}}\oh{\mathcal{V}_{0}^{q}(Du,Dh_{0})^{2}} \ dx +\tau_{q}^{n}\int_{B_{r}}H_{0}^{q}(Dh_{q}) \ dx\right\}\nonumber \\
\le& \tau^{n-\vartheta}_{q}\left(cr^{m}\tau_{q}^{\vartheta-n}+c\tau_{q}^{\vartheta}\right)\int_{B_{2r}}H(x,Du) \ dx,
\end{flalign}
where $c=c(\texttt{data}(\Omega_{0}),\vartheta)$. Again, we set $\rho=2r$ thus obtaining
\begin{flalign*}
\int_{B_{\tau_{q}\rho}}H(x,Du) \ dx \le \tau_{q}^{n-\vartheta}\left(c\rho^{m}\tau_{q}^{\vartheta-n}+\tau^{\vartheta}_{q}\right)\int_{B_{\rho}}H(x,Du) \ dx,
\end{flalign*}
where, as before, $\vartheta \in (0,n)$ is arbitrary. Take $\tau_{q}$ sufficiently small so that $c\tau^{\vartheta}_{q}<\frac{1}{2}$ and a threshold $R^{3}_{*}$, $0<R^{3}_{*}\le R^{2}_{*}$ such that $c(R^{3}_{*})^{m}\tau^{\vartheta-n}_{q}\le \frac{1}{2}$. Hence, for all $\rho \in (0,R^{3}_{*}]$ and all $\vartheta \in (0,n)$ we get
\begin{flalign}\label{h31}
\int_{B_{\tau_{q}\rho}}H(x,Du) \ dx \le \tau_{q}^{n-\vartheta}\int_{B_{\rho}}H(x,Du) \ dx.
\end{flalign}
Finally, select $\tau_{0}\in \left(0,\frac{1}{8}\right)$, assume $\texttt{ndeg}_{\beta}(B_{2r})$, where $r\le\frac{1}{2}R^{3}_{*}$. For $\vartheta \in (0,n)$, by \eqref{ndegc}, Poincar\'e's inequality, Proposition \ref{p0} with $\varphi = H_{0}$, \eqref{control} and \eqref{h25} we obtain
\begin{flalign}\label{h32}
\int_{B_{2\tau_{0}r}}H(x,Du) \ dx \le & c\int_{B_{4\tau_{0}r}}H_{0}\left(\frac{u-(u)_{4\tau_{0}r}}{4\tau_{0}r}\right) \ dx\le \oh{c\int_{B_{4\tau_{0}s}}H_{0}(Du) \ dx}\nonumber \\
\le &c\left\{\int_{B_{4\tau_{0}r}}\oh{\mathcal{V}_{0}(Du,Dh_{0})^{2}} \ dx+\int_{B_{4\tau_{0}r}}H_{0}^{s}(Dh_{0}) \ dx\right\} \nonumber \\
\le &c\left\{\int_{B_{4\tau_{0}r}}\oh{\mathcal{V}_{0}(Du,Dh_{0})^{2}} \ dx+\oh{\snr{B_{4\tau_{0}r}}\sup_{B_{4\tau_{0}r}}H_{0}(Dh_{0})}\right\} \nonumber \\
\le &c\left\{\int_{B_{r}}\oh{\mathcal{V}_{0}(Du,Dh_{0})^{2}}\ dx+\tau_{0}^{n}\int_{B_{r}}H_{0}^{s}(Dh_{0}) \ dx\right\}\nonumber\\ 
\le& \tau^{n-\vartheta}_{0}\left(cr^{m}\tau_{0}^{\vartheta-n}+c\tau_{0}^{\vartheta}\right)\int_{B_{2r}}H(x,Du) \ dx,
\end{flalign}
where $c=c(\texttt{data}(\Omega_{0}),\vartheta)$. Again, we set $\rho=2r$ thus obtaining
\begin{flalign*}
\int_{B_{\tau_{0}\rho}}H(x,Du) \ dx \le \tau_{0}^{n-\vartheta}\left(c\rho^{m}\tau_{0}^{\vartheta-n}+\tau^{\vartheta}_{0}\right)\int_{B_{\rho}}H(x,Du) \ dx,
\end{flalign*}
where, as before, $\vartheta \in (0,n)$ is arbitrary. Take $\tau_{0}$ sufficiently small so that $c\tau^{\vartheta}_{0}<\frac{1}{2}$ and a threshold $R^{4}_{*}$, $0<R^{4}_{*}\le R^{3}_{*}$ such that $c(R^{4}_{*})^{m}\tau^{\vartheta-n}_{0}\le \frac{1}{2}$. Hence, for all $\rho \in (0,R^{4}_{*}]$ and all $\vartheta \in (0,n)$ we get
\begin{flalign}\label{h34}
\int_{B_{\tau_{0}\rho}}H(x,Du) \ dx \le \tau_{0}^{n-\vartheta}\int_{B_{\rho}}H(x,Du) \ dx.
\end{flalign}

\cd{\emph{Step 2: double nested exit time and iteration}}\\
Now we are in position to develop the announced double nested exit time.\\
Take $B_{r}\Subset \Omega$ with $r \in (0,R_{*}]$, where $R_{*}=\min_{i \in \{1,2,3,4\}} \{R_{*}^{i}\}$ and consider $0<\rho<r$. For $\kappa \in \N\cup \{0\}$, we consider condition $\texttt{deg}(B_{\cd{2}\tau_{p}^{\kappa+1}r})$ and define the exit time index
\begin{flalign*}
t_{p}=\min\left\{\kappa \in \N\colon \ \texttt{deg}(B_{\cd{2}\tau_{p}^{\kappa+1}r}) \ \mathrm{fails}\right\}.
\end{flalign*}
For any $\kappa \in \{1,\cdots,t_{p}\}$ we apply repeatedly \eqref{h27} to obtain
\begin{flalign}\label{h35}
\int_{B_{\tau_{p}^{\kappa}r}}H(x,Du) \ dx \le \tau_{p}^{\kappa(n-\vartheta)}\int_{B_{r}}H(x,Du) \ dx.
\end{flalign}
The failure of $\texttt{deg}(B_{\cd{2}\tau_{p}^{\kappa+1}r})$ at $\kappa=t_{p}$, opens three different scenarios: either $\texttt{deg}_{\alpha}(B_{\cd{2}\tau_{p}^{t_{p}+1}r})$ or $\texttt{deg}_{\beta}(B_{\cd{2}\tau_{p}^{t_{p}+1}r})$ or directly $\texttt{ndeg}(B_{\cd{2}\tau_{p}^{t_{p}+1}r})$ is in force. Since the last condition is stable, and the first two are described by similar procedures, we shall focus on the occurrence of $\texttt{deg}_{\alpha}(B_{\cd{2}\tau_{p}^{t_{p}+1}r})$. Let us introduce a second exit time index
\begin{flalign*}
t_{s}=\min\left\{\kappa \in \N\colon \texttt{deg}_{\alpha}(B_{\cd{2}\tau_{s}^{\kappa+1}\tau_{p}^{t_{p}+1}r}) \ \mathrm{fails}\right\}.
\end{flalign*}
Iterating \eqref{h29} we obtain
\begin{flalign}\label{h36}
\int_{B_{\tau_{s}^{\kappa}\tau_{p}^{t_{p}+1}r}}H(x,Du) \ dx \le \tau_{s}^{\kappa(n-\vartheta)}\int_{B_{\tau_{p}^{t_{p}+1}r}}H(x,Du) \ dx.
\end{flalign}
If $\texttt{deg}_{\alpha}(B_{\cd{2}\tau_{s}^{\kappa+1}\tau_{p}^{t_{p}+1}r})$ fails at $\kappa=t_{s}$, the only chance we have is to look at $\texttt{ndeg}(B_{\cd{2}\tau_{s}^{t_{s}+1}\tau_{p}^{t_{p}+1}r})$. Condition $\texttt{ndeg}$ is stable, so we can iterate \eqref{h34} for $\kappa \in \N$, thus getting
\begin{flalign}\label{h37}
\int_{B_{\tau_{0}^{\kappa}\tau_{s}^{t_{s}+1}\tau_{p}^{t_{p}+1}r}}H(x,Du) \ dx \le \tau_{0}^{\kappa(n-\vartheta)}\int_{B_{\tau_{s}^{t_{s}+1}\tau_{p}^{t_{p}+1}r}}H(x,Du) \ dx.
\end{flalign}
Now we only need to fillet estimates \eqref{h34}-\eqref{h37}. For $0<\rho<r\le R_{*}$ we consider \oh{the following} \cd{five} cases.\\
\emph{\oh{Case (i):} $r>\rho\ge \tau_{p}^{t_{p}+1}r$}. Then there is $\bar{\kappa}\in \{0,\cdots,t_{p}\}$ such that $\tau_{p}^{\bar{\kappa}+1}r\le \rho< \tau_{p}^{\bar{\kappa}}r$. \oh{We obtain from \eqref{h35} that},
\begin{flalign}\label{h38}
\int_{B_{\rho}}H(x,Du) \ dx \le& \int_{B_{\tau_{p}^{\bar{\kappa}}r}}H(x,Du) \ dx \nonumber \\
\le&\tau_{p}^{\bar{\kappa}(n-\vartheta)}\int_{B_{r}}H(x,Du) \ dx\nonumber\\
\le &\tau_{p}^{(\bar{\kappa}+1)(n-\vartheta)}\tau_{p}^{\vartheta-n}\int_{B_{r}}H(x,Du) \ dx\le c\left(\frac{\rho}{r}\right)^{n-\vartheta}\int_{B_{r}}H(x,Du) \ dx,
\end{flalign}
where $c=c(\texttt{data}(\Omega_{0}),\vartheta)$.\\
\emph{\oh{Case (ii):} $\tau_{p}^{t_{p}+1}r> \rho\ge  \tau_{s}\tau_{p}^{t_{p}+1}r$}. \oh{We see that, by \eqref{h38}},
\begin{flalign}\label{h39}
\int_{B_{\rho}}H(x,Du) \ dx \le & \int_{B_{\tau_{p}^{t_{p}+1}r}}H(x,Du) \ dx\nonumber \\
\le &c\tau_{p}^{(t_{p}+1)(n-\vartheta)}\int_{B_{r}}H(x,Du) \ dx\nonumber \\
=&c(\tau_{s}\tau_{p}^{t_{p}+1})^{n-\vartheta}\tau_{s}^{\vartheta-n}\int_{B_{r}}H(x,Du) \ dx \le c\left(\frac{\rho}{r}\right)^{n-\vartheta}\int_{B_{r}}H(x,Du)\ dx,
\end{flalign}
with $c=c(\texttt{data}(\Omega_{0}),\vartheta)$.\\
\emph{\oh{Case (iii):} $\tau_{s}\tau_{p}^{t_{p}+1}r>\rho\ge \tau_{s}^{t_{s}+1}\tau_{p}^{t_{p}+1}r$}. So there is $\bar{\kappa}\in\{1, \cdots,t_{s}\}$ so that $\tau_{s}^{\bar{\kappa}}\tau_{p}^{t_{p}+1}r>\rho\ge \tau_{s}^{\bar{\kappa}+1}\tau_{p}^{t_{p}+1}r$. We have, by \eqref{h36} and \eqref{h38},
\begin{flalign}\label{h40}
\int_{B_{\rho}}H(x,Du) \ dx \le & \int_{B_{\tau_{s}^{\bar{\kappa}}\tau_{p}^{t_{p}+1}r}}H(x,Du) \ dx\nonumber \\
\le & \tau_{s}^{\bar{\kappa}(n-\vartheta)}\int_{B_{\tau_{p}^{t_{p}+1}r}}H(x,Du) \ dx \nonumber \\
\le & \tau_{s}^{(\bar{\kappa}+1)(n-\vartheta)}\tau_{s}^{\vartheta-n}\tau_{p}^{(t_{p}+1)(n-\vartheta)}\int_{B_{r}}H(x,Du) \ dx\le c\left(\frac{\rho}{r}\right)^{n-\vartheta}\int_{B_{r}}H(x,Du) \ dx,
\end{flalign}
where $c=c(\texttt{data}(\Omega_{0}),\vartheta)$.\\
\emph{\oh{Case (iv):}} $\tau_{s}^{t_{s}+1}\tau_{p}^{t_{p}+1}r>\rho\ge\tau_{s}^{t_{s}+1}\tau_{p}^{t_{p}+1}\tau_{0}r$. By \eqref{h40} we obtain
\begin{flalign}\label{xx}
\int_{B_{\rho}}H(x,Du) \ dx \le & \int_{B_{\tau_{s}^{t_{s}+1}\tau_{p}^{t_{p}+1}r}}H(x,Du) \ dx \nonumber \\
\le &c (\tau_{s}^{t_{s}+1}\tau_{p}^{t_{p}+1})^{n-\vartheta}\int_{B_{r}}H(x,Du) \ dx\nonumber\\
\le & c\tau_{0}^{\vartheta-n}(\tau_{0}\tau_{s}^{t_{s}+1}\tau_{p}^{t_{p}+1})^{n-\vartheta}\int_{B_{r}}H(x,Du) \ dx \le c\left(\frac{\rho}{r}\right)^{n-\vartheta}\int_{B_{r}}H(x,Du) \ dx,
\end{flalign}
with $c=c(\texttt{data}(\Omega_{0}),\vartheta)$.\\
\cd{\emph{Case (v): $\tau_{s}^{t_{s}+1}\tau_{p}^{t_{p}+1}\tau_{0}r>\rho>0$}. This condition renders a $\bar{\kappa}\in \N$ such that
$\tau_{0}^{\bar{\kappa}+1}\tau_{s}^{t_{s}+1}\tau_{p}^{t_{p}+1}\le \rho<\tau_{0}^{\bar{\kappa}}\tau_{s}^{t_{s}+1}\tau_{p}^{t_{p}+1}$. We then estimate, using \eqref{h37} and \eqref{xx},
\begin{flalign}\label{h41}
\int_{B_{\rho}}H(x,Du) \ dx \le & \int_{B_{\tau_{0}^{\bar{\kappa}}\tau_{s}^{t_{s}+1}\tau_{p}^{t_{p}+1}r}}H(x,Du) \ dx\nonumber \\
\le &\tau_{0}^{\bar{\kappa}(n-\vartheta)}\int_{B_{\tau_{s}^{t_{s}+1}\tau_{p}^{t_{p}+1}r}}H(x,Du) \ dx\nonumber \\
\le & c\tau_{0}^{\bar{\kappa}(n-\vartheta)}\tau_{s}^{(t_{s}+1)(n-\vartheta)}\tau_{p}^{(t_{p}+1)(n-\vartheta)}\int_{B_{r}}H(x,Du) \ dx \nonumber \\
\le & \tau_{0}^{\vartheta-n}\left(\frac{\rho}{r}\right)^{n-\vartheta}\int_{B_{r}}H(x,Du) \ dx=c\left(\frac{\rho}{r}\right)^{n-\vartheta}\int_{B_{r}}H(x,Du) \ dx,
\end{flalign}
where $c=c(\texttt{data}(\Omega_{0}),\vartheta)$.}\\

As mentioned before, the procedure is the same if, after $\texttt{deg}$ occurs $\texttt{deg}_{\beta}$ instead of $\texttt{deg}_{\alpha}$ and it is actually easier if, from $\texttt{deg}$ we jump directly to $\texttt{ndeg}$.\\
All in all we can conclude that, for all $0<\rho<r\le R_{*}$ and all $\vartheta \in (0,n)$ there holds
\begin{flalign}\label{h42}
\int_{B_{\rho}}H(x,Du) \ dx \le c\left(\frac{\rho}{r}\right)^{n-\vartheta}\int_{B_{r}}H(x,Du) \ dx,
\end{flalign}
with $c=c(\texttt{data}(\Omega_{0}),\vartheta)$. Now, if $r>R_{*}$ and $R_{*}\le \rho<r\le 1$ we get
\begin{flalign}\label{h43}
\int_{B_{\rho}}H(x,Du) \ dx \le & \left(\frac{\rho}{r}\right)^{n-\vartheta}\left(\frac{r}{\rho}\right)^{n-\vartheta}\int_{B_{r}}H(x,Du) \ dx\nonumber \\
\le &\left(\frac{r}{R_{*}}\right)^{n-\vartheta}\left(\frac{\rho}{r}\right)^{n-\vartheta}\int_{B_{r}}H(x,Du) \ dx \le c\left(\frac{\rho}{r}\right)^{n-\vartheta}\int_{B_{r}}H(x,Du) \ dx,
\end{flalign}
where $c=c(\texttt{data}(\Omega_{0}),\vartheta)$. Finally, if $\rho<R_{*}\le r\le 1$, we have
\begin{flalign}\label{h44}
\int_{B_{\rho}}H(x,Du) \ dx \le & c\left(\frac{\rho}{R_{*}}\right)^{n-\vartheta}\int_{B_{R_{*}}}H(x,Du) \ dx \nonumber \\
\le &c\left(\frac{r}{R_{*}}\right)^{\vartheta-n}\left(\frac{\rho}{r}\right)^{n-\vartheta}\int_{B_{r}}H(x,Du) \ dx \le c\left(\frac{\rho}{r}\right)^{n-\vartheta}\int_{B_{r}}H(x,Du) \ dx,
\end{flalign}
for $c=c(\texttt{data}(\Omega_{0}),\vartheta)$. Collecting estimates \eqref{h42}-\eqref{h44} we conclude that, for all $0<\rho<r\le 1$ there holds
\begin{flalign}\label{h45}
\int_{B_{\rho}}H(x,Du) \ dx \le c\left(\frac{\rho}{r}\right)^{n-\vartheta}\int_{B_{r}}H(x,Du) \ dx,
\end{flalign}
with $c=c(\texttt{data}(\Omega_{0}),\vartheta)$.

\section{Gradient continuity}
From \eqref{h45} and a standard covering argument, we can conclude that for every open subset $\Omega_{0}\Subset \Omega$ and $\kappa>0$ there exists a constant $c=c(\texttt{data}(\Omega_{0}),\kappa)$ such that
\begin{flalign}\label{h46}
\mint_{B_{r}}H(x,Du) \ dx \le cr^{-\kappa}
\end{flalign}
holds for every ball $B_{r}\Subset \Omega_{0}\Subset \Omega$, $r\le 1$. \cd{Now, if $h$ is any of the maps given by Lemma \ref{quahar} and $\oh{\tilde{H}}$ is one of the Young functions listed in \eqref{yfs} with $a_{0}=a_{i}(B_{2r})$ or $b_{0}=b_{i}(B_{2r})$, then, the theory in \cite{lieb} applies, thus rendering
\begin{flalign}\label{h47}
\mint_{B_{\rho}}\tilde{H}(Dh-(Dh)_{B_{\rho}})\ dx\le c\left(\frac{\rho}{r}\right)^{p\tilde{\nu}}\mint_{B_{r}}\tilde{H}(Dh) \ dx\le c\left(\frac{\rho}{r}\right)^{p\tilde{\nu}}\mint_{B_{r}}H(x,Du) \ dx,
\end{flalign}
where $c$ and $\tilde{\nu}$ depend at the most from $n,p,q,s$.}

Moreover, for $B_{r}\Subset \Omega_{0}$ with $0<r\le R_{*}$, where $R_{*}$ is the threshold radius introduced in the previous section, \oh{we obtain from Lemma \ref{quahar} and \eqref{h46} that}
\begin{flalign*}
\mint_{B_{r}}\oh{\widetilde{\mathcal{V}}(Du,Dh)^{2}} \ dx \le cr^{m}\mint_{B_{2r}}H(x,Du) \ dx \le cr^{m-\kappa}=cr^{\oh{\kappa_{0}}},
\end{flalign*}
by suitably fixing $0<\kappa<m$\oh{, where $\widetilde{\mathcal{V}}$ is the corresponding auxiliary function defined in \eqref{vfs} and the constant $c$ depends on $\texttt{data}(\Omega_{0})$.
Arguing exactly as in \cite[Section 10]{bcm}, we get
\begin{flalign}\label{h48}
\mint_{B_{r}}\tilde{H}(Du-Dh) \ dx \le cr^{\kappa_{1}}
\end{flalign}
for some positive exponent $\kappa_{1} = \kappa_{1}(\kappa_{0},n,p,q,s)$.}
\cd{In this case, $c=c(\texttt{data}(\Omega_{0}))$}. Now, for $0<\rho<r\le R_{*}$, by \eqref{h48}, the minimality of $h$, \eqref{h46} and \eqref{h47} we see that
\begin{flalign}\label{h49}
\mint_{B_{\rho}}\snr{Du-(Du)_{\rho}}^{p} \ dx \le & c\left \{\mint_{B_{\rho}}\snr{Dh-(Dh)_{\rho}}^{p} \ dx+\mint_{B_{\rho}}\snr{Du-Dh}^{p} \ dx\right\}\nonumber \\
\le & c\left\{\mint_{B_{\rho}}\oh{\tilde{H}}(Dh-(Dh)_{\rho}) \ dx+\left(\frac{r}{\rho}\right)^{n}\mint_{B_{r}}\oh{\tilde{H}}(Du-Dh) \ dx\right\}\nonumber \\
\le &c\left\{\cd{\left(\frac{\rho}{r}\right)^{p\tilde{\nu}}\mint_{B_{r}}H(x,Du) \ dx+\left(\frac{r}{\rho}\right)^{n}r^{m}\mint_{B_{2r}}H(x,Du) \ dx}\right\}\nonumber\\
\le& c\left(\rho^{p\tilde{\nu}}r^{-p\tilde{\nu}-\kappa}+\rho^{-n}r^{n+\kappa_{1}}\right),
\end{flalign}
with $c=c(\texttt{data}(\Omega_{0}),\kappa)$. Now, first notice that there is no loss of generality in supposing $p\tilde{\nu}\le 1$. Setting $\rho=r^{1+\frac{\kappa_{1}}{4n}}$ and $\kappa=\frac{\kappa_{1}p\tilde{\nu}}{8n}$ in \eqref{h49}, we easily obtain
\begin{flalign}\label{h50}
\mint_{B_{\rho}}\snr{Du-(Du)_{\rho}}^{p} \ dx \le c\rho^{\frac{\kappa_{1}p\tilde{\nu}}{16n}},
\end{flalign}
for all $\rho \in (0,R_{*})$, with $c=c(\texttt{data}(\Omega_{0}))$. Now, by the integral characterization of H\"older continuity due to Campanato and Meyers we can conclude that $Du \in C^{0,\nu}_{\mathrm{loc}}(\Omega,\mathbb{R}^{n})$ for $\nu=\frac{\kappa_{1}\tilde{\nu}}{16n}$. The full proof of Theorem \ref{T1} is still not complete, since $\nu$ depends on $\texttt{data}(\Omega_{0})$, while we announced that the H\"older continuity exponent of $Du$ depends only on $\texttt{data}$. So we will retain that, after a covering argument, $Du\in L^{\infty}_{\mathrm{loc}}(\Omega)$, therefore the non-uniform ellipticity of \eqref{hmvp} becomes immaterial. Now, for $B_{r}\Subset \Omega_{0}\Subset \Omega$, no matter what degeneracy condition holds, we compare $u$ to $h \in W^{1,H(\cdot)}(B_{r})$ solution to the Dirichlet problem
\begin{flalign}\label{final}
u+W^{1,H(\cdot)}_{0}(B_{r})\ni w\mapsto \min \int_{B_{r}}\snr{Dw}^{p}+a_{i}(B_{2r})\snr{Dw}^{q}+b_{i}(B_{2r})\snr{Dw}^{s} \ dx.
\end{flalign}
Notice that, for a functional like the one in \eqref{final}, the Bounded Slope Condition holds, see \cite{boubra}, so there exists $c=c(n,p,q,s,\nr{Du}_{L^{\infty}(B_{r})})$ such that  
\begin{flalign}\label{bsc}
\nr{Dh}_{L^{\infty}(B_{r})}\le c.
\end{flalign}
For simplicity, let us adopt the notation $H_{0}(z)=\snr{z}^{p}+a_{i}(B_{2r})\snr{z}^{q}+b_{i}(B_{2r})\snr{z}^{s}$. By strict convexity we obtain
\begin{flalign}\label{h51}
\mint_{B_{r}}&\cd{\mathcal{V}_{0}(Du,Dh)^{2}}\ dx\le  c \mint_{B_{r}}H_{0}(Du)-H_{0}(Dh) \ dx \nonumber \\
=&c\left\{ \mint_{B_{r}}H_{0}(Du)-H(x,Du) \ dx+\mint_{B_{r}}H(x,Du)-H(x,Dh) \ dx+\mint_{B_{r}}H(x,Dh)-H_{0}(Dh) \ dx\right\}\le cr^{\cd{\gamma}},
\end{flalign}
\cd{with $\gamma=\min\{\alpha,\beta\}$} and $c=c(p,q,s,[a]_{0,\alpha},[b]_{0,\beta},\nr{a}_{L^{\infty}(\Omega)},\nr{b}_{L^{\infty}(\Omega)},\nr{Du}_{L^{\infty}(\Omega_{0})})$. We got this last estimate by using \eqref{mathAB}, the boundedness of $\nr{Du}_{L^{\infty}_{\mathrm{loc}}(\Omega)}$ and \eqref{bsc}.
Now we jump back to \eqref{h49}, thus getting
\begin{flalign}\label{h54}
\mint_{B_{\rho}}\snr{Du-(Du)_{\rho}}^{p} \ dx\le &c\left\{\mint_{B_{\rho}}H(x,Du-Dh) \ dx +\left(\frac{\rho}{r}\right)^{p\tilde{\nu}}\mint_{B_{r}}H(x,Du) \ dx\right\}\le c\left(\rho^{-n}r^{n+\gamma}+\rho^{p\tilde{\nu}}r^{-p\tilde{\nu}}\right),
\end{flalign}
with $c=c(\texttt{data},\nr{Du}_{L^{\infty}(\Omega_{0})})$.
Equalizing in \eqref{h54} \cd{as we did to get \eqref{h49}}, we have
$$
\mint_{B_{\rho}}\snr{Du-(Du)_{\rho}}^{p} \ dx\le c\rho^{\nu p},
$$
\cd{with $\nu=\frac{\gamma\tilde{\nu}}{n+p\tilde{\nu}}$}. This means, by the integral characterization of H\"older continuity due to Campanato and Mayers, that $Du\in C^{0,\nu}_{\mathrm{loc}}(\Omega)$, \cd{and, recalling that $\tilde{\nu}=\tilde{\nu}(n,p,q,s)$, we see that now} $\nu=\nu(\texttt{data})$. This concludes the proof.


\begin{thebibliography}{50}
\bibitem{acefus1} E. Acerbi, N. Fusco, An approximation lemma for $W^{1,p}$ functions. Material Instabilities in Continuum Mechanics (Edinburgh, 1985-1986), 1-5. Oxford Science Publications, Oxford University Press, New York, (1988).
\bibitem{bcm1} P. Baroni, M. Colombo, G. Mingione, Nonautonomous functionals, borderline cases and related function classes. St. Petersburg Math. J. 27 347-379, (2016). 
\bibitem{bcm2} P. Baroni, M. Colombo, G. Mingione, Harnack inequalities for double phase functionals. Nonlinear Anal. 121 206-222, (2015). 
\bibitem{bcm} P. Baroni, M. Colombo, G. Mingione, Regularity for general functionals with double phase. Calc. Var. 57:62, (2018).
\bibitem{barlin} P. Baroni, C. Lindfors, The Cauchy-Dirichlet problem for a general class of parabolic equations. Ann. Inst. H. Poincar\'e Anal. Non Lin\'eaire 34, 593-624, (2017).
\bibitem{boubra} P. Bousquet, L. Brasco, Global Lipschitz continuity for minima of degenerate problems. Math. Ann. 366, 1403-1450, (2016).
\bibitem{byun1} S. S. Byun, Y. Youn, Riesz potential estimates for a class of double phase problems. J. Differential Equations 264   1263-1316, (2018). 
\bibitem{byunoh1} S. S. Byun, J. Oh, Global gradient estimates for the borderline case of double phase problems with BMO coefficients in nonsmooth domains. J. Differential Equations 263, 1643-1693, (2017).
\bibitem{byunoh2} S. S. Byun, J. Oh, Global gradient estimates for non-uniformly elliptic equations. Calc. Var. Partial Differential Equations 56:46,  (2017).
\bibitem{carkrispas} M. Carozza, J. Kristensen, A. Passarelli di Napoli, Regularity of minimizers of autonomous convex variational integrals. Ann. Sc. Norm. Super. Pisa Cl. Sci. vol. XIII, issue 5, (2014).
\bibitem{double} M. Colombo, G. Mingione, Regularity for double phase variational problems. Arch. Rational Mech. Anal. 215, 443-496, (2015).
\bibitem{colmin} M. Colombo, G. Mingione, Bounded minimizers of double phase variational integrals. Arch. Rational Mech. Anal. 218, 219-273, (2015).
\bibitem{colmin2} M. Colombo, G. Mingione, Calder\'on-Zygmund estimates and non-uniformly elliptic operators. J. Funct. Anal. 270, 1416-1478, (2016).

\bibitem{diestrver} L. Diening, B. Stroffolini, A. Verde, Everywhere regularity of functionals with $\varphi$-growth. Manuscripta Math. 129, 449-481, (2012).
\bibitem{ele} M. Eleuteri, H\"older continuity results for a class of functionals with non standard growth. Bollettino U. M. I., 129-157, (2004). 
\bibitem{elemarmas} M. Eleuteri, P. Marcellini, E. Mascolo: Lipschitz estimates for systems with ellipticity conditions at infinity. {\em Ann. Mat. Pura Appl. (IV)} 195, 1575-1603, (2016). 
\bibitem{sharp} L. Esposito, F. Leonetti, G. Mingione, Sharp regularity for functionals with $(p,q)$ growth. J. Differential Equations 204, 5-55, (2004).
\bibitem{ffm} I. Fonseca, J. Mal\'y, G. Mingione, Scalar minimizers with fractal singular sets. Arch. Ration. Mech. Anal. 172,  295-307, (2004). 
\bibitem{giagiu1} M. Giaquinta, E. Giusti, On the regularity of the minima of certain variational integrals. Acta Math., Vol. 148, 31-46, (1982). 
\bibitem{giusti} E. Giusti - Direct methods in the calculus of variations. World Scientific, (2003).
\bibitem{hasto} P. Harjulehto, P. H\"ast\"o, O. Toivanen, H\"older regularity of quasiminimizers under general growth conditions. Calc. Var. \& PDE 56, Paper No. 22, 26 pp, (2017).
\bibitem{hasto1} P. Harjulehto, P. H\"ast\"o, A. Karppinen, Local higher integrability of the gradient of a quasiminimizer under generalized Orlicz growth conditions. Nonlinear Anal., to appear.
\bibitem {KM1} \name[Kuusi, T.] \& \name[Mingione, G.]: Guide to nonlinear potential estimates. {\em Bull. Math. Sci.} 4, 1-82, (2014). 

\bibitem {KM2} \name[Kuusi, T.] \& \name[Mingione, G.]: Vectorial nonlinear potential theory. {\em J.~Europ.~Math.~Soc. (JEMS)} 20, 929-1004, (2018). 

\bibitem{leo} F. Leonetti, Higher integrability for minimizers of integral functionals with nonstandard growth. J. Differential Equations 112, 308-324 (1994).
\bibitem{lieb} G. M. Lieberman, The natural generalization of the natural condition of Ladyzhenskaya and Ural'tseva for elliptic equation. Comm. PDE 16, 311-361, (1991).
\bibitem {Manth} \name[Manfredi, J.J.]: Regularity of the gradient for a class of nonlinear possibly
degenerate elliptic equations. {\em Ph.D. Thesis}. University of
Washington, St. Louis, (1986).
\bibitem{mar1} P. Marcellini, Regularity of minimizers of integrals of the calculus of variations with nonstandard growth conditions. Arch. Ration. Mech. Anal. 105, 267-284, (1989).
\bibitem{mar2} P. Marcellini, Regularity and existence of solutions of elliptic equations with $p,q$-growth conditions. J. Differential Equations 90, 1-30, (1991).
\bibitem{mar3} P. Marcellini, Regularity for elliptic equations with general growth conditions. J. Differential Equations 105, 296-333, (1993).
\bibitem{dark} G. Mingione, Regularity of minima: an invitation to the Dark Side of the Calculus of Variations. Applications of Mathematics 51.4: 355-426, (2006). 
\bibitem{ok1} J. Ok, Regularity of $\omega$-minimizers for a class of functionals with non-standard growth. Calc. Var. Partial Differential Equations, 56:48, (2017). 
\bibitem{ok2} J. Ok, H\"older regularity for elliptic systems with non-standard growth, J. Funct. Anal. 274, 723-768, (2018). 
\end{thebibliography}
\end{document}